\documentclass[leqno]{amsart}
\usepackage[utf8]{inputenc}

\usepackage{amsmath}
\usepackage{amsthm}

\usepackage{enumerate}

\usepackage{mathtools}

\usepackage{stmaryrd}
 
\usepackage[colorlinks=false]{hyperref}

\usepackage[charter]{mathdesign}
\usepackage{xy}
\xyoption{all}

\makeatletter
\g@addto@macro\bfseries{\boldmath}
\makeatother

\numberwithin{equation}{section}

\theoremstyle{plain}
    \newtheorem{theorem}[equation]{Theorem}
    \newtheorem{lemma}[equation]{Lemma}
    \newtheorem{corollary}[equation]{Corollary}
    \newtheorem{proposition}[equation]{Proposition}

    \newtheorem*{theorem*}{Theorem}
    \newtheorem*{proposition*}{Proposition}
    \newtheorem*{corollary*}{Corollary}
    \newtheorem*{lemma*}{Lemma}
    \newtheorem*{conjecture*}{Conjecture}
    \newtheorem{definition-theorem}[equation]{Definition/Theorem}
    \newtheorem{definition-lemma}[equation]{Definition/Lemma}

\theoremstyle{definition}
    \newtheorem{definition}[equation]{Definition}
    \newtheorem{example}[equation]{Example}
    \newtheorem{examples}[equation]{Examples}

    \newtheorem{remark}[equation]{Remark}

    \newcommand{\N}{\mathbb{N}}
    \newcommand{\Z}{\mathbb{Z}}

\renewcommand{\phi}{\varphi}
\let\epsilon\varepsilon

\newcommand{\into}{\hookrightarrow}

\newcommand{\restrict}{\vert}

\newcommand{\id}{\mathrm{id}}
\newcommand{\ol}{\overline}

    \DeclareMathOperator{\Hom}{Hom}

    \DeclareMathOperator{\res}{res}
    
    \DeclareMathOperator{\ind}{ind}
    \DeclareMathOperator{\Ad}{Ad}
    \DeclareMathOperator{\Aut}{Aut}
    \DeclareMathOperator{\Rep}{Rep}
    \DeclareMathOperator{\Irr}{Irr}

    \DeclareMathOperator{\GL}{GL}
    
    \DeclareMathOperator{\pind}{i}
    \DeclareMathOperator{\pres}{r}
    \DeclareMathOperator{\supp}{supp}

    \DeclareMathOperator{\sign}{sign}
    \DeclareMathOperator{\Map}{Fun}

\newcommand{\Comp}{\operatorname{Comp}}
\newcommand{\Part}{\operatorname{Part}}

\newcommand{\Prim}{\operatorname{Prim}}
\newcommand{\prim}{\operatorname{prim}}

\newcommand{\Power}[1]{\mathcal P({#1})}

\newcommand{\orbits}[2]{{#1}\backslash {#2}}

\newcommand{\reg}{\operatorname{reg}}

\newcommand{\Chromatic}{\mathcal{G}}

\newcommand{\GraphAlg}{\mathcal{A}}

\newcommand{\Symreps}{\mathcal{S}}

\newcommand{\Sym}{\operatorname{Sym}}

\newcommand{\Zelev}[1]{\mathtt{#1}}

\newcommand{\triv}{\operatorname{triv}}

\newcommand{\bigast}{\mbox{\large$\ast$}}

\newcommand{\Set}{\operatorname{Set}}

\renewcommand{\k}{\Bbbk}

\newcommand{\inj}{\mathrm{inj}}

\newcommand{\Colour}{\operatorname{Col}}
 	 
\renewcommand{\setminus}{-}

\title[Combinatorial Hopf algebras from wreath products]{Combinatorial Hopf algebras from representations of families of wreath products}
\author{Tyrone Crisp and Caleb Kennedy Hill}
\address{Department of Mathematics \& Statistics, University of Maine.
5752 Neville Hall, Room 333.
Orono, ME 04469 USA}
\email{tyrone.crisp@maine.edu}
\email{caleb.hill@maine.edu}
\date{February 23, 2021}
 
\begin{document}

\begin{abstract}
We construct Hopf algebras whose elements are representations of combinatorial automorphism groups, by generalising a theorem of Zelevinsky on Hopf algebras of representations of wreath products. As an application we attach symmetric functions to representations of graph automorphism groups, generalising and refining Stanley's chromatic symmetric function.
\end{abstract}

\maketitle

\section{Introduction}

In this paper we construct Hopf algebras whose elements are linear representations of automorphism groups of certain combinatorial structures. We do this by generalising a theorem of Zelevinsky \cite[7.2]{Zelevinsky} on Hopf algebras of representations of wreath products $S_n\ltimes H^n$ to more general wreath products, and then applying Clifford theory to pass from wreath products to combinatorial automorphism groups. Let us illustrate and motivate the construction with an example.

The \emph{chromatic polynomial} of a finite, simple, undirected graph $\Gamma$ is the polynomial $\chi_\Gamma$ satisfying  $\chi_\Gamma(m) = \#\Colour_{\Gamma,m}$,  the number of proper $m$-colourings of $\Gamma$ (i.e., labellings of the vertices of $\Gamma$ by numbers $1,\ldots,m$ such that adjacent vertices have distinct labels). This much-studied graph invariant was first introduced by Birkhoff in \cite{Birkhoff}. A variation on $\chi_\Gamma$,  introduced in \cite{Hanlon} and further studied and generalised in \cite{Cameron-Jackson-Rudd,Cameron-Kayibi} under the name \emph{orbital chromatic polynomial}, counts $\#(\orbits{\Aut\Gamma}{\Colour_{\Gamma,m}})$, the number of orbits in $\Colour_{\Gamma,m}$ for the natural action of the automorphism group $\Aut\Gamma$. (In this paper the symbol `$\backslash$' will always denote an orbit space; for set-theoretic differences we shall write $X\setminus Y$.) To illustrate, the graphs
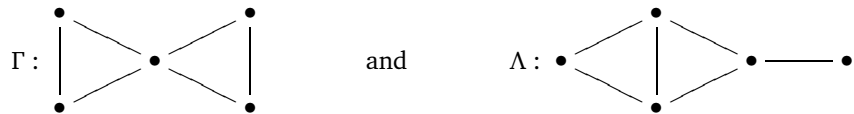
\[
\Gamma: \xygraph{ 
!{(0,0) }*+{\bullet}="a" 
!{(1,.5) }*+{\bullet}="b" 
!{(1,-.5) }*+{\bullet}="c" 
!{(-1,.5)}*+{\bullet}="d"
!{(-1,-.5)}*+{\bullet}="e" 
"a"-"b"-"c"-"a"-"d"-"e"-"a"
}
\qquad \qquad \textrm{and} \qquad \qquad
\Lambda: \xygraph{ 
!{(-1,0) }*+{\bullet}="a" 
!{(0,.5) }*+{\bullet}="b" 
!{(0,-.5) }*+{\bullet}="c" 
!{(1,0)}*+{\bullet}="d"
!{(2,0)}*+{\bullet}="e" 
"e"-"d"-"b"-"a"-"c"-"d"
"b"-"c"
}
\]
have $\chi_{\Gamma}=\chi_{\Lambda}$, but $\# (\orbits{\Aut\Gamma}{\Colour_{\Gamma,3}})=3$ while $\# (\orbits{\Aut\Lambda}{\Colour_{\Lambda,3}})=6$. (This pair of graphs is taken from \cite[Figure 1]{Stanley-chromatic}.)

One can generalise and refine these invariants using finite harmonic analysis. Letting $\k^{\Colour_{\Gamma,m}}$ be the permutation representation of $\Aut\Gamma$ corresponding to the action of $\Aut\Gamma$ on $\Colour_{\Gamma,m}$ ($\k$ being some algebraically closed field of characteristic zero), we consider for each finite-dimensional $\k$-linear representation $\gamma$ of $\Aut\Gamma$ the intertwining number 
\[
\chi_{\Gamma,\gamma}(m)\coloneqq \dim\Hom_{\Aut\Gamma}(\k^{\Colour_{\Gamma,m}},\gamma) = \frac{1}{\#\Aut\Gamma} \sum_{g\in \Aut\Gamma} \operatorname{ch}_\gamma(g) \cdot \#\Colour_{\Gamma,m}^g
\]
where $\operatorname{ch}_\gamma$ is the character of the representation $\gamma$, and $\Colour_{\Gamma,m}^g$ is the set of colourings fixed by the automorphism $g$. Hanlon observed in \cite[Theorem 2.1]
{Hanlon} that the cardinalities of these fixed sets are themselves chromatic polynomials,  and so each $\chi_{\Gamma,\gamma}$ is a polynomial. Putting $\gamma=$ the regular representation gives the chromatic polynomial, while putting $\gamma=$ the trivial one-dimensional representation gives the orbital chromatic polynomial. 

It is clear from the definition that the map $\gamma\mapsto \chi_{\Gamma,\gamma}$ is additive, and so decomposing the regular representation into irreducibles gives 
\[
\chi_\Gamma= \sum_{\gamma\in \Irr(\Aut\Gamma)} (\dim \gamma) \chi_{\Gamma,\gamma}
\]
where the sum runs over the set of isomorphism classes of irreducible representations. The polynomials $\chi_{\Gamma,\gamma}$, and the decomposition of $\chi_\Gamma$ that they afford, deserve closer study. In particular, one would like to understand how these polynomials behave with respect to unions and decompositions of graphs, and it is at this point that Hopf algebras enter the picture.

The use of Hopf algebras to study assembly/disassembly constructions is well established, both in combinatorics (see, e.g., \cite{Joni-Rota}, \cite{Schmitt-HACS}, \cite{Schmitt-IHA}, \cite{Aguiar-Mahajan}, \cite{Grinberg-Reiner}) and in representation theory (see, e.g., \cite{Geissinger}, \cite{Zelevinsky}, \cite{vanLeeuwen}, \cite{Aguiar-et-al}, \cite{Shelley-Abrahamson}). Let us recall two examples of particular relevance to the present discussion.

Our first example is the  \emph{Hopf algebra of graphs} $\Chromatic$ \cite{Schmitt-HACS}: this is the free abelian group with basis the set of isomorphism classes of finite simple graphs; with multiplication given by disjoint union of graphs; and with comultiplication given by partitions into pairs of subgraphs. The character $\Chromatic\to \Z$ given by $\Gamma\mapsto \chi_{\Gamma}(1)$ induces, as shown in \cite{ABS}, a morphism of Hopf algebras $\Chromatic\to \Sym_{\Z}$ into the Hopf algebra of symmetric functions with integer coefficients; this map sends $\Gamma$ to the \emph{chromatic symmetric function} $X_\Gamma$ introduced by Stanley in \cite{Stanley-chromatic}, and the chromatic polynomial $\chi_\Gamma$ can be recovered from $X_\Gamma$ by specialisation: 
$\chi_\Gamma(m) = X_\Gamma(1^m)$. For a discussion of how the Hopf-algebra point of view illuminates certain properties of $\chi_\Gamma$ and $X_\Gamma$, see \cite[7.3]{Grinberg-Reiner}.

Our second example of a Hopf algebra is the \emph{Hopf algebra of representations of the symmetric groups} $\Symreps$ \cite[\S6]{Zelevinsky}: this is the free abelian group with basis the set of isomorphism classes of irreducible representations of the symmetric groups $S_n$ (for all $n\geq 0$); with multiplication given by induction of representations, and comultiplication by restriction of representations, along the standard inclusions $S_k\times S_l\into S_{k+l}$. The character $\Symreps\to\Z$ sending each trivial representation to $1$ and all other irreducible representations to $0$ induces---again via the technology of \cite{ABS}---a morphism of Hopf algebras $\Symreps \to \Sym_{\Z}$ that is in fact an isomorphism; this is the well-known \emph{Frobenius characteristic} map, sending the irreducible representation of $S_n$ associated with a partition $\lambda$ of $n$ to the Schur function $s_\lambda$; see, e.g., \cite[7.18]{Stanley-EC2}.

In Section \ref{sec:graphs} of this paper, as an instance of the general results established in Sections \ref{sec:Young} and \ref{sec:Hopf}, we construct a Hopf algebra $\GraphAlg$ that simultaneously generalises both of the above examples. The underlying additive group of our Hopf algebra $\GraphAlg$ is free abelian with basis $\bigsqcup_{\Gamma}\Irr(\Aut\Gamma)$, the set of isomorphism classes of irreducible representations of the automorphism groups of finite simple graphs (modulo graph isomorphisms). The multiplication/comultiplication in $\GraphAlg$ are given by combining union/decomposition of graphs with induction/restriction of representations. The map sending $\gamma\in \Irr(\Aut\Gamma)$ to $\chi_{\Gamma,\gamma}(1)\in \Z$ induces a homomorphism of Hopf algebras $\GraphAlg\to \Sym_{\Z}$, thus associating a symmetric function $X_{\Gamma,\gamma}$ to each finite graph $\Gamma$ and each finite-dimensional representation $\gamma$ of $\Aut\Gamma$. The polynomial $\chi_{\Gamma,\gamma}$ defined above is recovered from $X_{\Gamma,\gamma}$ by specialisation. 

The subalgebra of $\GraphAlg$ corresponding to graphs with no edges is isomorphic to the Hopf algebra $\Symreps$ of representations of the symmetric groups, and the map $\Symreps \to \Sym_{\Z}$ obtained by restricting our map $\GraphAlg\to \Sym_{\Z}$ is the Frobenius characteristic. On the other hand, the map sending a graph $\Gamma$ to the regular representation of $\Aut\Gamma$ gives an embedding of Hopf algebras $\Chromatic\into \GraphAlg$, and the decomposition of the regular representation into irreducibles yields an equality of symmetric functions 
\[
X_\Gamma = \sum_{\gamma\in \Irr(\Aut\Gamma)}(\dim\gamma)X_{\Gamma,\gamma}.
\] 
We thus obtain a refinement of the polynomial invariants $\chi_{\Gamma,\gamma}$ by symmetric functions, generalising Stanley's refinement of $\chi_\Gamma$ by $X_\Gamma$; and we obtain identities among the $X_{\Gamma,\gamma}$ (and, by specialisation, among the $\chi_{\Gamma,\gamma}$) for varying $\Gamma$ and $\gamma$ from the fact that the map $\gamma\mapsto X_{\Gamma,\gamma}$ is a homomorphism of Hopf algebras. The further study of the graph invariants $X_{\Gamma,\gamma}$ and $\chi_{\Gamma,\gamma}$ will be taken up in future work.

We now describe the connection with wreath products, still in the example of graph colourings. For each $n\geq 0$ let $E_n$ denote the set of two-element subsets of $\{1,\ldots,n\}$. The symmetric group $S_n$ acts in a natural way on $E_n$, and hence on the group $\Map(E_n, S_2)$ of functions $E_n\to S_2$, and on the set $\Irr(\Map(E_n,S_2))$ of irreducible representations of this abelian group. The $S_n$-orbits in $\Irr(\Map(E_n,S_2))$ can be identified with the isomorphism classes of graphs with $n$ vertices, in such a way that the stabiliser of a point in $\Irr(\Map(E_n,S_2))$ is equal to the automorphism group of the corresponding graph. Clifford theory (as explained in this context in \cite[Section 4.3]{James-Kerber}) then yields an identification
\begin{equation}\label{eq:Clifford-intro}\tag{$\star$}
\bigsqcup_{\Gamma} \Irr(\Aut\Gamma) \xrightarrow{\cong} \bigsqcup_{n\geq 0} \Irr(S_n\ltimes \Map(E_n,S_2))
\end{equation}
between the basis for $\GraphAlg$ and the irreducible representations of the wreath product groups $S_n\ltimes \Map(E_n,S_2)$. 

The representation theory of wreath products can be quite complicated: indeed, the bijection \eqref{eq:Clifford-intro} shows that classifying the irreducible representations of the groups $S_n\ltimes \Map(E_n,S_2)$ for all $n$ means classifying all finite simple graphs up to isomorphism \emph{and} classifying the irreducible representations of all finite groups (since every finite group is, as shown by Frucht  \cite{Frucht}, the automorphism group of a graph). 

There is, however, one aspect of the representation theory of the wreath products $S_n\ltimes  \Map(E_n,S_2)$ that is rather more easily understood: namely, the way in which the representations of these groups fit together for different $n$. Generalising work of Zelevinsky \cite{Zelevinsky}, who considered wreath products of the form $S_n\ltimes H^n$, we prove that for each suitable family of $S_n$-sets $Y_n$, and for each finite group $H$, the free abelian group with basis $\bigsqcup_{n\geq 0} \Irr(S_n\ltimes \Map(Y_n,H))$ can be given a natural Hopf-algebra structure. In fact we obtain three different (in general) Hopf algebra structures: one a positive self-adjoint Hopf algebra as in \cite{Zelevinsky}, and another dual pair of non-PSH but connected, commutative, and cocommutative Hopf algebras. (In the situation studied by Zelevinsky these three Hopf algebras are all identical.) Our Hopf algebras come equipped moreover with a canonical $\Z$-valued character; as shown by Aguiar, Bergeron, and Sottile \cite{ABS} this is equivalent to admitting a canonical homomorphism  into  the Hopf algebra of symmetric functions. Putting $Y_n=E_n$ and $H=S_2$ yields the Hopf algebra $\GraphAlg$ of representations of graph automorphisms, and the symmetric functions $X_{\Gamma,\gamma}$. More examples of this kind are possible: for example, letting $Y_n$ be the set of all nonempty subsets of $\{1,\ldots,n\}$ has the effect of replacing graphs by hypergraphs; letting $Y_n$ be the set of ordered two-element subsets gives directed graphs; and replacing $S_2$ by another group $H$ has the effect of introducing labellings of the edges of our (hyper)graphs by the nontrivial irreducible representations of $H$.

The paper is organised as follows. In Section \ref{sec:Young} we first describe the families of sets $Y_n$ that go into our construction. The definition is easily stated: we consider endofunctors on the category of finite sets and injective maps that preserve the empty set and preserve intersections; then $Y_n$ is the value of such a functor on the set $\{1,\ldots,n\}$. We then define induction and restriction functors between the representations of the  wreath products $S_n\ltimes \Map(Y_n,H)$ for varying $n$; we consider both the standard induction/restriction functors, and a variant of these functors similar to the parabolic induction/restriction functors from the representation theory of reductive groups. These functors become, in Section \ref{subsec:R-Hopf}, the multiplication and comultiplication maps in our Hopf algebras. In Section \ref{subsec:Clifford} we apply Clifford theory to yield a second description of our Hopf algebras in terms of representations of automorphisms of certain combinatorial structures. In Section \ref{subsec:basic} we show that each of our Hopf algebras contains a sub-Hopf-algebra of representations of the base group $\Map(H,{Y_n})$, an example being the subalgebra $\Chromatic$ of $\GraphAlg$; and in Section \ref{subsec:zeta} we compute the canonical maps from our Hopf algebras to $\Sym_{\Z}$, under the assumption that the coefficient group $H$ is abelian. Section \ref{sec:graphs} then presents in more detail the example of $\GraphAlg$ that we outlined above.

Our constructions bear a resemblance to known constructions of Hopf algebras from species, such as those described in \cite{Schmitt-HACS} and \cite{Aguiar-Mahajan} for example, although as far as we are aware the main construction that we study here has not previously appeared in the literature beyond the special cases $Y_N=\emptyset$ and $Y_N=N$.  There is however one concrete point of overlap between our construction and \cite{Schmitt-HACS}: if $H$ is abelian then one of the sub-Hopf-algebras that we construct in Section \ref{subsec:basic}---namely, the subalgebra generalising the subalgebra $\Chromatic$ of $\GraphAlg$---is isomorphic to the Hopf algebra of a coherent exponential species as defined in \cite{Schmitt-HACS}; see Proposition \ref{prop:HACS}.

\subsection*{Acknowledgements} The first author thanks Ehud Meir and Uri Onn for early discussions on the topic of this project, the idea for which first arose during our joint work on \cite{CMO1}. Some of the results presented here also appear in the second author's MA thesis, submitted to the University of Maine in Spring 2020.

\section{Young sets and wreath products}\label{sec:Young}

\subsection{Young sets}

In this section we define the combinatorial objects from which we shall construct families of wreath product groups.

We let $\Set$ denote the category of finite sets, while $\Set^\times$ and $\Set^{\inj}$ denote the subcategories of bijective maps and injective maps, respectively. A functor $Y:\Set^\inj \to \Set^{\inj}$ thus assigns to each finite set $N$ a finite set $Y_N$, and to each injective map of finite sets $w:N\to M$ an injective map $Y_w:Y_N\to Y_M$. When $M=N$ we obtain an action of the symmetric group $S_N$ of $N$ on the set $Y_N$.
We will often write $w:Y_N\to Y_M$, instead of  $Y_w$; and in the case where $w$ is the inclusion of $N$ as a subset of $M$ we shall omit $w$ from the notation entirely and regard $Y_N$ as a subset of $Y_M$.

\begin{definition}\label{def:Young-set}
A \emph{Young set} is a  functor $Y: \Set^{\inj}  \to \Set^{\inj}$ satisfying $Y_\emptyset = \emptyset$ and $Y_K\cap Y_L=Y_{K\cap L}$ for all pairs of subsets $K,L$ of the same finite set $N$.
\end{definition}

The name `Young set' was chosen because these sets, and the families of groups to which they give rise, play an analogous role in our construction to that played by the Young subgroups in the representation theory of the symmetric groups.

\begin{examples}\label{examples:Young-set}
In most of these examples we describe the action of the functor on objects only, the action on morphisms being the obvious one.
\begin{enumerate}[\rm(1)]
\item The empty example: $Y_N=\emptyset$ for all sets $N$. We denote this example by $\emptyset$.
\item The basic example: $Y_N=N$, the  identity functor. We denote this example by $\id$.
\item Products, coproducts, and composites: if $Y$ and $Y'$ are Young sets, then  so are the product $(Y\times Y')_N\coloneqq  Y_N\times Y'_N$; the coproduct $(Y\sqcup Y')_N\coloneqq Y_N\sqcup Y'_N$; and the composite $(Y\circ Y')_N\coloneqq Y_{Y'_N}$. For instance,  for each fixed $m\geq 1$ we obtain a Young set $\id^m:N\mapsto N^m$ (where $N^m = N\times  \cdots \times N$).
\item Quotients: let $Y$ be a Young set equipped with an action of a fixed group $G$ by natural transformations; then the quotient by the $G$-action yields a Young set $Y/G$ with $(Y/G)_N = Y_N/G$. For instance, taking $Y_N=N^m$, on which $G=S_m$ acts by permuting coordinates, we obtain the Young set  of unordered $m$-tuples of elements of $N$.
\item Subsets and multisets: the functor $N\mapsto \{\text{nonempty subsets of $N$}\}$ is a Young set, as is the functor $N\mapsto \{\text{$r$-element subsets of $N$}\}$ for each $r\geq 1$. More generally, let $m$ be a positive integer and let $\xi:\{1,\ldots,m\} \to \Power{\N}$ be a function ($\mathcal P$ denotes the power set and $\N=\{0,1,\ldots\}$). To these data we associate the Young set
\[
Y_N\coloneqq \left\{ f: N\to \{0,1,\ldots,m\}\ \left|\ \begin{gathered} \exists n\in N\text{ with $f(n)\neq 0$, and } \\ \text{$\#f^{-1}(i)\in \xi(i)$ for every $i\geq 1$}\end{gathered}\right.\right\}.
\]
The map $Y_K\to Y_N$ associated to an injection of sets $K\into N$ is given by extending functions by $0$. Taking $m=1$ and $\xi(1)=\{1,2,3,\ldots\}$ gives the Young set of nonempty subsets of $N$; taking $m=1$ and $\xi(1)=\{r\}$ gives the $r$-element subsets; while taking $m\geq 2$ gives multisets of elements of $Y_N$ with up to $m$ repetitions of each element, with the function $\xi$ imposing restrictions on the number of times each allowed multiplicity occurs.
\item Permutations: $Y_N=S_N\setminus \{\id_N\}$, the set of nontrivial permutations of $N$. (Recall that `$\setminus$' means the set-theoretic difference.) This is a Young set with the action on morphisms given by assigning to each injective map of sets $w:K\to N$ and each permutation $g$ of $K$ the permutation $w_*g$ of $N$ defined by
\[
w_*g (n)\coloneqq \begin{cases} wg(k) & \text{if }n=w(k)\in w(K)\\ n & \text{otherwise.}\end{cases}
\]
\end{enumerate}
\end{examples}

\subsection{Families of wreath products from Young sets}

We begin with some generalities on wreath products; see, e.g., \cite[Chapter 4]{James-Kerber} for more information. 

Let $W$ be a finite group acting on a finite set $Y$, and let $H$ be another finite group. The set $\Map(Y,H)$ of functions $Y\to H$ is made into a group by setting $(f_1 f_2)(y)\coloneqq f_1(y) f_2(y)$. The group $W$ acts on $\Map(Y,H)$ by $(w  f)(y) \coloneqq f(w^{-1}y)$, for $w\in W$, $f\in \Map(Y,H)$, and $y\in Y$. The {wreath product} of $W$ with $H$ over $Y$ is defined to be the semidirect product $W\ltimes \Map(Y,H)$. Thus as a set $W\ltimes \Map(Y,H)= W\times \Map(Y,H)$, with group operation 
\[
(w_1,f_1)(w_2,f_2) = \big(w_1 w_2, (w_2^{-1}f_1) f_2 \big).
\]

The maps $w\mapsto (w,1)$ and $f\mapsto (1,f)$ identify $W$ and $\Map(Y,H)$ with subgroups of $W\ltimes \Map(Y,H)$. The \emph{support} of a function $f\in \Map(Y,H)$ is defined by $\supp(f)=\{y\in Y\ |\ f(y)\neq 1_H\}$. For each subset $Y'$ of $Y$ we regard $\Map(Y',H)$ as a subgroup of $\Map(Y,H)$, namely the subgroup $\{f:Y\to H\ |\ \supp(f) \subseteq Y'\}$. If $W'$ is a subgroup of $W$ such that $Y'$ is $W'$-invariant, then the embeddings $W'\subseteq W\subseteq W\ltimes \Map(Y,H)$ and $\Map(Y',H)\subseteq \Map(Y,H) \subseteq W\ltimes \Map(Y,H)$ give an embedding $W'\ltimes \Map(Y',H) \subseteq W\ltimes \Map(Y,H)$. If $W_1$ acts on $Y_1$, and $W_2$ acts on $Y_2$, then there is an obvious isomorphism
\[
(W_1\ltimes \Map(Y_1,H)) \times (W_2\ltimes \Map(Y_2, H)) \xrightarrow{\cong} (W_1\times W_2)\ltimes \Map(Y_1\sqcup Y_2,H)
\]

So much for generalities. We shall study the representation theory of the family of wreath product groups 
\[
G_N(Y,H) \coloneqq S_N \ltimes \Map(Y_N,H)
\]
associated to a Young set $Y$ and an auxiliary finite group $H$.

\begin{examples}\begin{enumerate}[\rm(1)]
\item $G_\emptyset(Y,H)=S_\emptyset \ltimes \Map(\emptyset,H)$ is always the trivial group.
\item $G_N(\emptyset,H)=S_N$ for every $H$, and $G_N(Y,S_1)=S_N$ for every $Y$.
\item Writing $[n]=\{1,2,\ldots,n\}$, we have $G_{[n]}(\id,H)=S_n\ltimes H^n$, the standard wreath product studied in \cite[\S7]{Zelevinsky}. For example, $G_{[n]}(\id, S_2)$ is the hyperoctahedral group, whose representations were first worked out by Young in \cite{Young}.
\item $G_{[n]}(\id^2,C_p)$ (where $C_p$ is the cyclic group of prime order $p$) is isomorphic to the group of those invertible $n\times n$ matrices over the ring $\Z/p^2\Z$ that are congruent, modulo $p$, to a permutation matrix. An isomorphism is given by identifying $S_n$ with the group of permutation matrices in $\GL_n(\Z/p^2\Z)$; identifying $C_p$ with the additive group $p\Z/p^2\Z$; and then identifying the group $\Map([n]^2,C_p)$ with the kernel of the reduction-modulo-$p$ map $\GL_n(\Z/p^2\Z)\to \GL_n(\Z/p\Z)$ by sending a function $f$ to the matrix $1 + [f(i,j)]_{(i,j)\in [n]^2}$. One can think of this group $G_{[n]}(\id^2,C_p)$ as a simplified model of $\operatorname{GL}_n(\Z/p^2\Z)$. The induction functors that we shall consider below are, in this example, the analogues of the functors used to study the representations of $\GL_n(\Z/p^k\Z)$ in \cite{CMO1}, \cite{CMO2}, and \cite{CMO3}. 
\end{enumerate}
\end{examples}

\subsection{Young subgroups of $G_N(Y,H)$}

We are going to define analogues, in the wreath products $G_N(Y,H)=S_N\ltimes \Map(Y_N,H)$, of the Young subgroups of the symmetric groups. We begin with some notation related to set partitions.

A \emph{weak partition} of a finite set $N$ is a finite multiset  $\lambda=(L_i\ |\ i\in I)$ of subsets $L_i\subseteq N$, called the \emph{blocks} of $\lambda$, having $L_i\cap L_j=\emptyset$ for $i\neq j$, and $N=\bigcup_{i\in I} L_i$. One or more of the $L_i$ may be empty; a weak partition with no empty blocks is called a \emph{partition}. We denote by $\Part^w_N$ the set of  weak partitions of $N$, and by $\Part_N$ the set of partitions.

To each $\lambda=(L_i)\in \Part^w_N$ we associate the Young subgroup $S_\lambda\subseteq S_N$ consisting of those permutations that leave invariant each of the blocks $L_i$.  There is an obvious isomorphism $\prod_{i} S_{L_i} \cong S_{\lambda}$.
Inserting or removing $\emptyset$s from a weak partition does not change the Young subgroup. 

For $\lambda,\mu\in \Part^w_N$ we write $\lambda\leq \mu$ to mean that each block of $\mu$ is a union of blocks of $\lambda$. This is not a partial order, but it restricts to a partial order on partitions. We have $\lambda \leq \mu$ if and only if $S_\lambda \subseteq S_\mu$.

Given $\lambda,\mu\in \Part^w_N$ we let $\lambda \wedge \mu\in \Part^w_N$ be the weak partition whose blocks are the intersections $L_i\cap M_j$ of the blocks of $\lambda$ and $\mu$. We have $S_{\lambda\wedge\mu}=S_{\lambda}\cap S_{\mu}$. 

Each bijective map of sets $w:N\to M$ induces a bijective map $\Part^w_N\to\Part^w_M$, by applying $w$ to each block of each partition. In particular, the group $S_N$ acts on $\Part^w_N$. The group $S_\lambda$ fixes the partition $\lambda$, but the isotropy group of $\lambda$ may be larger than $S_\lambda$. For each bijection $w:N\to M$  and each $\lambda \in \Part^w_N$ we have ${}^wS_\lambda = S_{w\lambda}$, where ${}^wS_\lambda\coloneqq wS_\lambda w^{-1}\subset S_M$.

We now return to our groups $G_N(Y,H)$. The Young set $Y$ and auxiliary group $H$ will be fixed throughout this section, so we drop them from the notation and just write $G_N$.
 
For each weak partition $\lambda=(L_i)$ of $N$ the subsets $Y_{L_i}$ of $Y_N$ are pairwise disjoint, by our assumption that $Y_K\cap Y_L=Y_{K\cap L}$. We denote by $Y_\lambda \coloneqq \bigsqcup_{i} Y_{L_i}\subseteq Y_N$ the  union of these subsets. We have $Y_\lambda \cap Y_\mu = Y_{\lambda \wedge\mu}$ for all $\lambda,\mu\in \Part^w_N$, and $Y_\lambda \subseteq Y_\mu$ for all $\lambda \leq \mu\in \Part^w_N$. 

When $Y=\id$ we have $Y_\lambda=Y_N$ for every $\lambda$ and every $N$, but in general $Y_\lambda \subsetneq Y_N$. The functoriality of $Y$ ensures that for each bijection $w:N\to M$ and each $\lambda \in \Part^w_N$ we have $wY_\lambda = Y_{w\lambda}$ as subsets of $Y_M$. In particular, $Y_\lambda$ is $S_\lambda$-invariant.

We now have subgroups $S_\lambda \subseteq S_N$ and $\Map(Y_\lambda,H) \subseteq \Map(Y_N,H)$. We  take the semidirect product of these groups to obtain a subgroup $G_\lambda\subseteq G_N$:
\[
G_\lambda(Y,H) \coloneqq S_\lambda \ltimes \Map(Y_\lambda,H).
\]
We have a canonical isomorphism of groups $\prod_i G_{L_i}\cong G_\lambda$, coming from the corresponding isomorphisms $\prod_i S_{L_i}\cong S_\lambda$ and $\prod_i \Map(Y_{L_i},H) \cong \Map( \sqcup_i Y_{L_i}, H) = \Map(Y_\lambda,H)$.

Let us list some properties of the groups $G_\lambda$; all of these are immediate consequences of the corresponding facts about the groups $S_\lambda$ and the sets $Y_\lambda$.

\begin{lemma}\label{lem:G-properties} Let $N$ and $M$ be finite sets.
\begin{enumerate}[\rm(1)]

\item If $\lambda \leq \mu\in \Part^w_N$ then $G_\lambda \subseteq G_\mu$.
\item For all $\lambda,\mu\in \Part^w_N$ we have $G_\lambda \cap G_\mu = G_{\lambda \wedge \mu}$.
\item For each bijective map $w:N\to M$ and each $\lambda \in \Part^w_N$ we have ${}^wG_\lambda = G_{w\lambda}$.\hfill\qed
\end{enumerate}
\end{lemma}

For each $\lambda \leq \mu \in \Part^w_N$ we consider the following subgroups of $G_\mu$:
\[
P_\lambda^\mu \coloneqq S_\lambda \ltimes \Map(Y_{\mu}, H )\qquad\text{and}\qquad U_\lambda^\mu \coloneqq \Map({Y_\mu \setminus Y_\lambda},H).
\]
Let us list some properties of these groups; all of these follow easily from the definitions.

\begin{lemma}\label{lem:P-U-properties}
Let $N$ and $M$ be finite sets.
\begin{enumerate}[\rm(1)]
\item If $\lambda \leq \mu$ in $\Part^w_N$ then  the group  $U_\lambda^\mu$ is normalised by $G_\lambda$, and the map  $G_\lambda \times U_\lambda^\mu  \to P_\lambda^\mu$ given by multiplication in the group $G_\mu$ is a bijection; thus $P_\lambda^\mu$ is the internal semidirect product $G_\lambda \ltimes U_\lambda^\mu$. 
\item If $w:N\to M$ is a bijection of sets then for all $\lambda\leq \mu\in \Part^w_N$ we have ${}^wP_\lambda^\mu = P_{w\lambda}^{w\mu}$ and ${}^wU_\lambda^\mu= U_{w\lambda}^{w\mu}$.
\item For all $\lambda,\mu\in \Part^w_N$ we have $G_\lambda \cap U_\mu^N = U_{\lambda \wedge\mu}^\lambda$.\hfill\qed
\end{enumerate}
\end{lemma}

We next recall some terminology from \cite{Zelevinsky}: for each $\lambda \in \Part^w_N$ we say that a subgroup $G\subseteq G_N$ is \emph{decomposable} with respect to $(G_\lambda,U_\lambda^N)$ if the intersection $P_\lambda^N\cap G$ decomposes as the semidirect product $(G_\lambda\cap G)\ltimes ( U_\lambda^N\cap G)$. 

\begin{lemma}\label{lem:decomposability}
Let $\lambda,\mu\in \Part^w_N$ be weak partitions of a finite set $N$.
 Each of the subgroups $P_\mu^N$, $G_\mu$, and $U_\mu^N$  of $G_N$ is decomposable with respect to $(G_\lambda,U_\lambda^N)$.
\end{lemma}

\begin{proof}
Compute as follows:
\[
\begin{aligned}
P_\lambda^N \cap P_\mu^N   =  (S_\lambda \cap S_\mu) \ltimes \Map({Y_N},H)   & = \left(S_{\lambda\wedge\mu}\ltimes \Map({Y_\lambda},H) \right)\ltimes \Map({Y_N\setminus Y_\lambda},H) \\ & = (G_\lambda\cap P_\mu^N)\ltimes (U_\lambda^N\cap P_\mu^N),
\end{aligned}
\]
\[
\begin{aligned}
 P_\lambda^N\cap G_\mu  = (S_\lambda \cap S_\mu) \ltimes \Map({Y_\mu},H)  & = \left( S_{\lambda\wedge\mu}\ltimes \Map({Y_{\lambda\wedge\mu}},H)\right)\ltimes \Map({Y_\mu\setminus Y_{\lambda\wedge\mu}},H) \\ &  = (G_\lambda \cap G_\mu) \ltimes (U_\lambda^N \cap G_\mu),
 \end{aligned}
 \]
 and
 \[
 \begin{aligned}
 P_\lambda^N\cap U_\mu^N = U_\mu^N = \Map({Y_N\setminus Y_\mu},H) & = \Map({Y_\lambda \setminus Y_{\lambda\wedge\mu}},H) \times \Map({Y_N \setminus (Y_\lambda \cup Y_\mu)},H) \\ &= (G_\lambda \cap U_\mu^N) \ltimes (U_\lambda^N\cap U_\mu^N). &\qedhere
\end{aligned}
\]
\end{proof}

\subsection{Induction and restriction functors}

We continue to fix a Young set $Y$ and a finite group $H$, and write $G_N$ for the wreath product $G_N(Y,H)=S_N\ltimes \Map({Y_N},H)$. We also fix an algebraically closed field $\k$ of characteristic zero.  For each finite group $G$ we let $\Rep(G)$ denote the category whose objects are the finite-dimensional $\k$-linear representations of $G$, and whose morphisms are the $G$-equivariant linear maps. 

\begin{definition} 
For all ordered pairs of weak partitions $\lambda\leq \mu\in \Part^w_N$ we consider the following functors:
\begin{enumerate}[\rm(1)]
\item Let
$
\pind_\lambda^\mu : \Rep(G_\lambda) \to \Rep(G_\mu)
$
be the functor of inflation from $G_\lambda$ to $P_\lambda^\mu$ (i.e., let $U_\lambda^\mu$ act trivially), followed by induction from $P_\lambda^\mu$ to $G_\mu$. 
\item Let 
$
\pres^\mu_\lambda: \Rep(G_\mu) \to \Rep(G_\lambda)
$
be the functor that sends each $G_\mu$-representation $V$ to the $G_\lambda$-invariant subspace $V^{U_\lambda^\mu}$ of $U_\lambda^\mu$-fixed vectors.
\item Let
$
\res^\mu_\lambda: \Rep(G_\mu)\to \Rep(G_\lambda)
$
be the usual restriction functor.
\end{enumerate}
\end{definition}

For the Young set $Y_N=N$ the groups $U_\lambda^\mu$ are trivial, and so there is no difference between the functors $\pres^\mu_\lambda$ and $\res^\mu_\lambda$. In general $\pres_\lambda^\mu$ is a subfunctor of $\res^\mu_\lambda$, and the inclusion can be proper.

\begin{lemma}\label{lem:adjoints}
The functors $\pind_\lambda^\mu$ and $\pres_\lambda^\mu$ are two-sided adjoints to one another.
\end{lemma}

\begin{proof}
Since the groups in question are all finite and $\k$ has characteristic zero, induction from $P_\lambda^\mu$ to $G_\mu$ is a two-sided adjoint to restriction from $G_\mu$ to $P_\lambda^\mu$, while inflation from $G_\lambda$ to $P_\lambda^\mu$ is a two-sided adjoint to the functor of $U_\lambda^\mu$-invariants.
\end{proof}

\begin{lemma} \label{lem:transitivity}
For each ordered triple $\lambda \leq \mu \leq \nu \in \Part^w_N$ there are natural isomorphisms of functors
\[
\pind_\lambda^\nu \cong \pind_\mu^\nu \pind_\lambda^\mu \qquad \text{and}\qquad \pres_\lambda^\nu \cong \pres^\mu_\lambda \pres^\nu_\mu.
\]
\end{lemma}

\begin{proof}
By the uniqueness of adjoint functors it will suffice to prove the assertion about $\pres^\nu_\lambda$. For each representation $V$ of $G_\nu$ we have $\pres^\nu_\lambda(V)= V^{U^\nu_\lambda}$, while $\pres^\mu_\lambda \pres^\nu_\mu (V)= \left( V^{U^\nu_\mu}\right)^{U^\mu_\lambda}$. The decomposition of sets $Y_\nu \setminus Y_\lambda = (Y_\nu \setminus Y_\mu) \sqcup (Y_\mu\setminus Y_\lambda)$ leads to an internal direct-product decomposition of groups $U^\nu_\lambda = U^\nu_\mu \times U^\mu_\lambda$, showing that $\pres^\nu_\lambda(V)$ and $\pres^\mu_\lambda\pres^\nu_\mu(V)$ are in fact equal as $G_\lambda$-invariant subspaces of $V$.
\end{proof}

For each bijection of sets $w:N\to M$ we have an equivalence 
\begin{equation}\label{eq:Adw-R}
\Ad_w:\Rep(G_N)\xrightarrow{\rho\mapsto \rho(w^{-1}\cdot w)} \Rep(G_M),
\end{equation}
where $\rho(w^{-1} \cdot w)$ means the map $g \mapsto \rho(w^{-1} g w)$.

\begin{lemma}\label{lem:w-i-r}
For each bijection of sets $w:N\to M$ and each ordered pair of weak partitions $\lambda \leq \mu \in \Part^w_N$ we have natural isomorphisms of functors
\[
\Ad_w \pind_\lambda^\mu \cong \pind_{w\lambda}^{w\mu} \Ad_w,\quad \Ad_w\pres^\mu_\lambda \cong \pres^{w\mu}_{w\lambda} \Ad_w,\quad \text{and}\quad \Ad_w \res^\mu_\lambda \cong \res^{w\mu}_{w\lambda} \Ad_w.
\]
\end{lemma}

\begin{proof}
Once again by the uniqueness of adjoints it suffices to consider the functors $\pres$ and $\res$. The statement about $\res$ is clearly true, while the assertion about $\pres$ follows from the equality ${}^wU_\lambda^\mu=U_{w\lambda}^{w\mu}$ observed in Lemma \ref{lem:P-U-properties}.
\end{proof}

We conclude this section by establishing Mackey-type formulas for the composition of our restriction and induction functors. The formulas are instances of \cite[Theorem A3.1]{Zelevinsky}. 

\begin{proposition}\label{prop:Mackey}
For each finite set $N$ and each pair of weak partitions $\lambda,\mu\in \Part^w_N$ we have isomorphisms of functors 
\[
\begin{aligned}
 \pres^N_{\lambda}\pind_{\mu}^N & \cong \bigoplus_{S_\lambda w S_\mu \in S_\lambda \backslash S_N/S_\mu} \pind^{\lambda}_{\lambda \wedge w\mu} \Ad_w \pres^{\mu}_{w^{-1}\lambda \wedge \mu}\quad \text{and} \\
 \res^N_{\lambda}\pind_{\mu}^N & \cong \bigoplus_{S_\lambda w S_\mu \in S_\lambda \backslash S_N/S_\mu}  \pind^{\lambda}_{\lambda\wedge w\mu} \Ad_w \res^{\mu}_{w^{-1}\lambda \wedge \mu}.
 \end{aligned}
 \]
\end{proposition}

\begin{proof}
We first establish the formula for $\pres^N_{\lambda}\pind_{\mu}^N$ by applying \cite[Theorem A3.I]{Zelevinsky} to the following choices of groups:
\[
\Zelev{G}=G_N,\quad  \Zelev{M}=G_{\mu},\quad  \Zelev{U}=U_{\mu}^N,\quad  \Zelev{P}=P_{\mu}^N,\quad  \Zelev{N}=G_{\lambda},\quad \Zelev{V}=U_{\lambda}^N, \quad  \Zelev{Q}=P_{\lambda}^N.
\]
($\Zelev{G}$, $\Zelev{P}$, etc., designate the objects denoted by those letters in \cite[A3]{Zelevinsky}, while $G$, $P$, etc., refer to objects defined in this paper.)

The characters ${\theta}$ and $\psi$ appearing in \cite{Zelevinsky} are here taken to be trivial. 
The double-coset space $\Zelev{Q}\backslash \Zelev{G}/\Zelev{P}$ is computed thus:
\[
P_{\lambda}^N \backslash G_N / P_{\mu}^N = (S_\lambda \ltimes \Map({Y_N},H))\backslash (S_N\ltimes \Map({Y_N},H)) / (S_{\mu}\ltimes \Map(Y_N,H)) \cong S_\lambda \backslash S_N / S_{\mu}.
\]
For each $w\in S_N$ the groups ${}^w\Zelev{P}=P_{w\mu}^N$, ${}^w\Zelev{M}=G_{w\mu}$, and ${}^w\Zelev{U}=U_{w\mu}^N$ are decomposable with respect to $(\Zelev{N}, \Zelev{V})= (G_\lambda, U_\lambda^N)$ by Lemma \ref{lem:decomposability}. Likewise ${}^w\Zelev{Q}$, ${}^w\Zelev{N}$, and ${}^w\Zelev{V}$ are decomposable with respect to $(\Zelev{M},\Zelev{U})$, so the hypothesis (D) of \cite[p. 168]{Zelevinsky} is satisfied. 

Lemmas \ref{lem:G-properties} and \ref{lem:P-U-properties} yield the following identifications of the groups $\Zelev{M}'$, $\Zelev{N}'$, etc., appearing on \cite[p.~168]{Zelevinsky}:
\[
\Zelev{M}'= G_{w^{-1}\lambda \wedge \mu},\quad \Zelev{N}' = G_{\lambda \wedge w\mu},\quad \Zelev{V}'=U_{w^{-1}\lambda\wedge\mu}^\mu,\quad \Zelev{U}'=U_{\lambda \wedge w\mu}^{\lambda}.
\]
Having made these identifications, an application of \cite[Theorem A3.1]{Zelevinsky} gives the stated formula for $\pres^N_{\lambda}\pind_{\mu}^N$.  

The proof of the formula for $\res^N_{\lambda}\pind_{\mu}^N$ is very similar: we now take 
\[
\Zelev{G}=G_N,\quad  \Zelev{M}=G_{\mu},\quad  \Zelev{U}=U_{\mu}^n,\quad  \Zelev{P}=P_{\mu}^N,\quad  \Zelev{N}=\Zelev{Q}=G_{\lambda},\quad \Zelev{V}=\{1\}.
\]
We still have $\Zelev{Q}\backslash \Zelev{G}/\Zelev{P}\cong S_\lambda \backslash S_N/ S_\mu$, and the decomposability hypothesis is again satisfied by virtue of Lemma \ref{lem:decomposability}. We now have
\[
\Zelev{M}'= G_{w^{-1}\lambda \wedge \mu},\quad \Zelev{N}' = G_{\lambda\wedge w\mu},\quad \Zelev{V}'=\{1\},\quad \Zelev{U}' =U_{\lambda \wedge w\mu}^{\lambda},
\]
and the formula from \cite[Theorem A3.1]{Zelevinsky} becomes in this case the proposed  formula for $\res^N_{\lambda}\pind_{\mu}^N$.
\end{proof}

\section{Hopf algebras associated to Young sets}\label{sec:Hopf}

\subsection{Construction of the Hopf algebras}\label{subsec:R-Hopf}

We continue to consider the family of groups $G_N=G_N(Y,H)= S_N\ltimes \Map(Y_N,H)$ associated to a Young set $Y$ and a finite group $H$. 

For each finite group $G$ we let $R(G)$ denote the Grothendieck group of $\Rep(G)$. Thus $R(G)$ is a free abelian group with basis $\Irr(G)$, the set of isomorphism classes of irreducible $\k$-linear representations of $G$. 

For all pairs $K,L$ of finite sets we write $G_{K,L}$ to mean the Young subgroup of $G_{K\sqcup L}$ associated to the weak partition with blocks $K$ and $L$. The isomorphism $G_K\times G_L\xrightarrow{\cong} G_{K,L}$ and the bijection  $\Irr(G_K)\times \Irr(G_L)\xrightarrow{(V,V')\mapsto V\otimes_{\k} V'} \Irr(G_K\times G_L)$ yield a canonical isomorphism $R(G_K)\otimes_{\Z} R(G_L) \xrightarrow{\cong} R(G_{K,L})$ which we shall frequently invoke without further comment. 

We define $\mathcal R_{Y,H}$, or $\mathcal R$ for short, to be the free abelian group 
\[
\mathcal R = \big(\bigoplus_{N\in \Set} R(G_N)\big)_{\Set^\times}
\]
where the subscript indicates that we take coinvariants for the groupoid of set bijections; that is to say, we impose the relation $\rho = \Ad_w \rho$ whenever $\rho$ and $w$ are as in \eqref{eq:Adw-R}. Thus $\mathcal R$ is a free abelian group  with   basis $\left(\bigsqcup_{N}\Irr(G_N)\right)_{\Set^\times}$. We grade $\mathcal R$ so that $R(G_N)$ sits in degree $\#N$.

We consider the following graded, $\Z$-linear maps:
\begin{enumerate}[$\square$]
\item multiplication: $m:{\mathcal R}\otimes_{\Z} {\mathcal R}\to {\mathcal R}$ defined as the direct sum of the maps 
\[
R(G_K)\otimes_{\Z} R(G_L) \xrightarrow{\cong} R(G_{K,L}) \xrightarrow{\pind_{K,L}^{K\sqcup L}} R(G_{K\sqcup L}).
\]
\item comultiplication: $\Delta:{\mathcal R}\to {\mathcal R}\otimes_{\Z} {\mathcal R}$ defined on $\rho\in R(G_N)$ by
\begin{equation}\label{eq:Delta-R-definition}
\Delta(\rho) = \sum_{S_N(K)\in \orbits{S_N}{\Power{N}}} \res^N_{K,K^c} \rho.
\end{equation}
Here the sum is over a set of representatives for the $S_N$-orbits of subsets of $N$, and $K^c=N\setminus K$. The representation $\res^N_{K,K^c}\rho$ of $G_{K,K^c}$ is regarded as an element of $\mathcal R\otimes_{\Z}\mathcal R$ via the canonical isomorphism $R(G_{K,K^c})\cong R(G_K)\otimes_{\Z} R(G_{K^c})$.
\item another comultiplication: $\delta:{\mathcal R}\to {\mathcal R}\otimes_{\Z} {\mathcal R}$ defined on $\rho\in R(G_N)$ by
\[
\delta(\rho) = \sum_{S_N(K)\in \orbits{S_N}{\Power{N}}} \pres^N_{K,K^c} \rho.
\]
This is to be understood in the same way as \eqref{eq:Delta-R-definition}.
\item unit: $e:\Z\to {\mathcal R}$ defined by setting $e(1)\coloneqq \triv_{G_\emptyset}$, the unique element of $\Irr(G_\emptyset)$.
\item counit: $\epsilon:{\mathcal R}\to \Z$ defined by setting $\epsilon(\triv_{G_\emptyset})=1$, and $\epsilon(\rho)=0$ for all other irreducible representations $\rho$.
\end{enumerate}
Note that Lemma \ref{lem:w-i-r} ensures that $m$, $\Delta$, and $\delta$ are well-defined on $\Set^\times$-coinvariants.

\begin{theorem}\label{thm:PSH}
Fix a Young set $Y$ and a finite group $H$. 
\begin{enumerate}[\rm(1)]
\item The maps $m$, $\Delta$, $e$, and $\epsilon$ make ${\mathcal R}_{Y,H}$ into a graded, connected, commutative, and cocommutative Hopf algebra over $\Z$. We denote this Hopf algebra by ${\mathcal R}^{\Delta}_{Y,H}$.
\item The  maps $m$, $\delta$, $e$, and $\epsilon$ and the basis $\big(\bigsqcup_{N} \Irr(G_N)\big)_{\Set^\times}$ make ${\mathcal R}_{Y,H}$ into a \emph{PSH algebra}: a graded, connected, positive, self-adjoint Hopf algebra over $\Z$ (cf.~\cite[1.4]{Zelevinsky}). We denote this PSH algebra by ${\mathcal R}^{\delta}_{Y,H}$.
\end{enumerate}
\end{theorem}

\begin{example}\label{example:non-isomorphic}
$\mathcal R^{\delta}_{\emptyset,1}$ and $\mathcal R^{\Delta}_{\emptyset,1}$ are both the Hopf algebra of representations of the symmetric groups, which is isomorphic to the Hopf algebra $\Sym_{\Z}$ of symmetric functions with $\Z$ coefficients; see \cite[\S5, \S6]{Zelevinsky}. Both $\mathcal R^{\delta}_{\id,H}$ and $\mathcal R^{\Delta}_{\id,H}$ are the Hopf algebras constructed in \cite[\S7]{Zelevinsky}. Unlike in these examples, the Hopf algebras $\mathcal R^\delta_{Y,H}$ and $\mathcal R^\Delta_{Y,H}$ are generally distinct (as  Hopf algebras with distinguished bases): see Remark \ref{rem:graph-PSH} for an example.
\end{example}

\begin{remark}
Taking the dual of the Hopf algebra ${\mathcal R}^{\Delta}_{Y,H}$ gives a third  Hopf-algebra structure on ${\mathcal R}_{Y,H}$, in which the multiplication is given by the usual induction functors $\ind_{G_{K,L}}^{G_{K\sqcup L}}$, while the comultiplication is given by the functors $\pres^{K\sqcup L}_{K,L}$. Of course, the PSH algebra ${\mathcal R}^{\delta}_{Y,H}$ is its own dual.
\end{remark}

\begin{proof}[Proof of Theorem \ref{thm:PSH}]
The proof that ${\mathcal R}^\Delta$ and ${\mathcal R}^\delta$ satisfy the Hopf axioms---that is, the axioms listed in \cite[1.3]{Zelevinsky}---is similar for the two cases, and both are similar to the case of $Y =\emptyset$ and $H=\{1\}$ established in \cite[6.2]{Zelevinsky}. 
Most of the axioms follow in a very straightforward way from the basic properties of the functors $\pind$, $\pres$, and $\res$ observed in the previous section; for example, the associativity of multiplication follows from Lemma \ref{lem:transitivity}. The compatibility between multiplication and comultiplication---that is, the commutativity of the diagram
\[
\xymatrix@C=60pt{
{\mathcal R}\otimes_{\Z}  {\mathcal R} \ar[r]^-{m} \ar[d]_-{\Delta\otimes \Delta} & {\mathcal R} \ar[r]^-{\Delta} & {\mathcal R}\otimes_{\Z} {\mathcal R} \\
{\mathcal R}\otimes_{\Z} {\mathcal R}\otimes_{\Z} {\mathcal R} \otimes_{\Z} {\mathcal R} \ar[rr]^-{w\otimes x\otimes y\otimes z\longmapsto w\otimes y\otimes x\otimes z} & & {\mathcal R}\otimes_{\Z} {\mathcal R} \otimes_{\Z} {\mathcal R}\otimes_{\Z} {\mathcal R} \ar[u]_-{m\otimes m}
}
\]
and of the corresponding diagram for $\delta$---follows from Proposition \ref{prop:Mackey} just as in \cite[A3.2]{Zelevinsky}. 

Thus ${\mathcal R}^\Delta$ and ${\mathcal R}^\delta$ are connected, graded Hopf algebras. The proofs of parts (1) and (2) diverge at this point; let us handle (2) first.

We must verify that ${\mathcal R}^\delta$ satisfies the additional axioms from \cite[1.4]{Zelevinsky}; again, the argument closely follows that of \cite[6.2]{Zelevinsky}. The fact that  $m$ and $\delta$ are adjoints to one another with respect to the inner products induced by our choice of basis follows from  the fact that $\pind_{K,L}^{K\sqcup L}$ and $\pres^{K\sqcup L}_{K,L}$ are adjoint functors  (Lemma \ref{lem:adjoints}). The positivity of all of the structure maps relative to our chosen basis follows immediately from the fact that all of these maps are defined via functors between representation categories. This completes our proof of part (2).

To complete the proof of part (1) we must show that the Hopf algebra ${\mathcal R}^\Delta$ is commutative and cocommutative. We have ${\mathcal R}^\Delta={\mathcal R}^\delta$ as algebras, and the PSH algebra ${\mathcal R}^\delta$ is automatically commutative (see \cite[Proposition 1.6]{Zelevinsky}). So we are left to prove that ${\mathcal R}^\Delta$ is cocommutative, which amounts to the assertion that for all subsets $K\subseteq N$ the diagram
\begin{equation}\label{eq:cocommutative-proof-1}
\xymatrix@R=5pt{
& R(G_{K,K^c})\ar[r]^-{\cong} & R(G_K)\otimes_{\Z} R(G_{K^c}) \ar[dd]^-{\text{flip}}\\
R(G_{N})\ar[ur]^-{\res^{N}_{K,K^c}} \ar[dr]_-{\res^{N}_{K^c,K}} & & \\
& R(G_{K^c,K}) \ar[r]^-{\cong} & R(G_{K^c})\otimes_{\Z} R(G_K)
}
\end{equation}
commutes. But this is obvious because $G_{K,K^c}$ and $G_{K^c,K}$ are the same subgroup of  $G_N$.
\end{proof}

\subsection{Application of Clifford theory}\label{subsec:Clifford}

Let $Y$ be a Young set, let $H$ be a finite group, and let $G_N=G_N(Y,H)=S_N\ltimes \Map(Y_N,H)$ as before. An application of Clifford theory  gives a description of the set $\Irr(G_N)$  in terms of orbits and isotropy groups for the action of $S_N$ on the set $\Irr(\Map(Y_N,H))$. In this section we shall briefly recall how this correspondence works (referring to \cite[Section 4.3]{James-Kerber}, for instance, for the details); and we then use this correspondence to give another description of the Hopf algebras $\mathcal R^{\Delta}_{Y,H}$ and $\mathcal R^{\delta}_{Y,H}$.

Fix a set $\widehat{H}$ of representatives for the isomorphism classes of irreducible representations of $H$. For each finite set $N$ and each function $F\in \Map(Y_N,\widehat{H})$ we let $\pi_F\in \Irr(\Map(Y_N,H))$ be the representation defined by 
\[
\pi_F(f)\coloneqq \bigotimes_{y\in Y_N} F(y)\left(f(y)\right) \in \GL\left(\bigotimes_{y\in Y_N} V_{F(y)}\right) \quad (f\in \Map(Y_N,H)).
\]
Here $V_{F(y)}$ is the vector space underlying the representation $F(y)\in \widehat{H}$, and $F(y)\left(f(y)\right)$ is the linear map $V_{F(y)}\to V_{F(y)}$ by which  $f(y)\in H$ acts under the representation $F(y)$.
The map $F\mapsto \pi_F$ is a bijection $\Map(Y_N,\widehat{H}) \to \Irr(\Map(Y_N,H))$.

For each bijection of sets $w:N\to M$ and each $F\in \Map(Y_N,\widehat{H})$ we define $wF\in \Map(Y_M,\widehat{H})$ by $wF(y)\coloneqq F(w^{-1}y)$. Setting $M=N$ gives an action of the group $S_N$ on the function space $\Map(Y_N,\widehat{H})$. For each $F\in \Map(Y_N,\widehat{H})$ we define
\[
\Aut F \coloneqq \{w\in S_N\ |\ wF=F\} \quad \text{and}\quad G_F\coloneqq (\Aut F)\ltimes \Map(Y_N,H)\subseteq G_N.
\]
For each representation $\gamma$ of $\Aut F$ we let $\gamma\ltimes \pi_F$ be the representation of $G_F=\Aut F\ltimes \Map(Y_N,H)$ on the tensor product vector space $V_\gamma\otimes_{\k} \bigotimes_{y\in Y_N} V_{F(y)}$, where $\Map(Y_N,H)$ acts trivially on $V_\gamma$ and by $\pi_F$ on $\bigotimes_y V_{F(y)}$, and where $\Aut F$ acts on $V_\gamma$ by $\gamma$ and on $\bigotimes_y F_y$ by permuting the factors: 
\[
w \bigotimes_{y\in Y_N} v_y \coloneqq \bigotimes_{y\in Y_N} v_{w^{-1}y} \qquad (v_y\in V_{F(y)}).
\]
This is well defined because $V_{F(y)}$ and $V_{F(w^{-1}y)}$ are the same vector space. 

For each bijection of sets $w:N\to M$ and each $F\in \Map(Y_N,\widehat{H})$ we have ${}^w\Aut F=\Aut wF$, giving an equivalence
\begin{equation}\label{eq:Adw-M}
\Ad_w: \Rep(G_N)\xrightarrow{\gamma\mapsto \gamma(w^{-1}\cdot w)} \Rep(G_M).
\end{equation}

Clifford theory, in this case, says that the maps
\[
\Phi_F: \Irr ( \Aut F ) \to \Irr(G_N),\qquad \Phi_F(\gamma)\coloneqq \ind_{G_F}^{G_N} (\gamma\ltimes \pi_F ),
\]
defined for each $F\in \Map(Y_N,\widehat{H})$, assemble into a bijective map
\begin{equation}\label{eq:Clifford-Mackey-bijection}
 \Bigg( \bigsqcup_{\substack{N\in \Set, \\ F\in \Map(Y_N,\widehat{H})}} \Irr(\Aut F) \Bigg)_{\Set^\times} \xrightarrow{\quad \bigsqcup \Phi_F\quad}  \Big( \bigsqcup_{N\in \Set} \Irr(G_N)\Big)_{\Set^\times}.
\end{equation}
As before, the subscript $\Set^\times$ indicates the quotient space for the actions \eqref{eq:Adw-R} and \eqref{eq:Adw-M} of the groupoid $\Set^\times$.

Now consider 
\begin{equation}\label{eq:M-definition}
\mathcal M_{Y,H} \coloneqq \Big( \bigoplus_{\substack{N\in \Set \\ F\in \Map(Y_N,\widehat{H})}}  R( \Aut F)\Big)_{\Set^\times}
\end{equation}
which  is a free abelian group with graded basis
\begin{equation}\label{eq:M-basis}
 \Big( \bigsqcup_{\substack{N\in \Set, \\ F\in \Map(Y_N,\widehat{H})}} \Irr(\Aut F) \Big)_{\Set^\times} .
 \end{equation}

The bijection of bases from  \eqref{eq:Clifford-Mackey-bijection} gives an isomorphism of  groups  $\Phi:\mathcal M_{Y,H} \xrightarrow{\cong} \mathcal R_{Y,H}$, and hence Theorem \ref{thm:PSH} furnishes $\mathcal M_{Y,H}$ with two Hopf-algebra structures. Our purpose in this section is to describe these structures explicitly. Since $Y$ and $H$ will be fixed we shall henceforth drop them from the notation when convenient, writing $\mathcal M$ for $\mathcal M_{Y,H}$ and $\mathcal R$ for $\mathcal R_{Y,H}$.

For each pair of finite sets $K,L$ we have an $S_{K,L}$-equivariant embedding
\begin{equation}\label{eq:F-sqcup}
\Map(Y_K,\widehat{H}) \times \Map(Y_L,\widehat{H}) \xrightarrow{\cong} \Map({Y_K\sqcup Y_L},\widehat{H}) = \Map(Y_{K,L},\widehat{H}) \into \Map({Y_{K\sqcup L}},\widehat{H})
\end{equation}
where the last arrow is defined by extending each function $F:Y_{K,L}\to \widehat{H}$ to a function $Y_{K\sqcup L}\to \widehat{H}$ by defining $F(y)=\triv_H$ for all $y\in Y_{K\sqcup L}\setminus (Y_K\sqcup Y_L)$. Here $\triv_H$ denotes the one-dimensional trivial representation of $H$. We shall denote the embedding \eqref{eq:F-sqcup} by $(F_K,F_L)\mapsto F_K\sqcup F_L$.

The standard embedding $S_K\times S_L \into S_{K\sqcup L}$ restricts to an embedding $\Aut F_K \times \Aut {F_L}\into \Aut (F_K\sqcup F_L)$, and so we have an induction functor
\[
\ind_{\Aut F_K \times \Aut F_L}^{\Aut (F_K\sqcup F_L)} : \Rep(\Aut F_K\times \Aut F_L) \to \Rep(\Aut(F_K\sqcup F_L)).
\]
On the free abelian group $\mathcal M$ we define a graded multiplication $\mathcal M\otimes_{\Z} \mathcal M\to \mathcal M$ as the direct sum of the maps
\begin{equation}\label{eq:M-mult-definition}
R(\Aut F_K)\otimes_{\Z} R(\Aut F_L) \xrightarrow{\cong} R(\Aut F_K \times \Aut F_L) \xrightarrow{\ind} R(\Aut(F_K\sqcup F_L)).
\end{equation}
The transitivity of induction ensures that $\mathcal M$ becomes an associative graded algebra with this multiplication; the unit is the trivial representation of the trivial automorphism group $\Aut F_\emptyset$, where $F_\emptyset\in \Map({Y_\emptyset},\widehat{H})$ is the empty function.

\begin{proposition}\label{prop:Clifford-Mackey-multiplication}
The map $\Phi : \mathcal M\to \mathcal R$ is an isomorphism of graded algebras.
\end{proposition}

\begin{proof}
We already know that $\Phi$ is a graded isomorphism of abelian groups. The map $\Phi_\emptyset: R(\Aut F_\emptyset) \to R(G_\emptyset)$ sends the trivial representation to the trivial representation, which is to say, $\Phi$ sends the unit of $\mathcal M$ to the unit of $\mathcal R$.

 The proposition thus amounts to the assertion that for all finite sets $K,L$, and all $F_K\in  \Map(Y_K,\widehat{H})$ and $F_L\in   \Map(Y_L,\widehat{H})$, the diagram
\begin{equation}\label{eq:Clifford-Mackey-multiplication-diag}
\xymatrix@C=40pt{
\Rep(G_{K,L}) \ar[r]^-{\pind_{K,L}^{K\sqcup L}} & \Rep(G_{K\sqcup L}) \\
\Rep(\Aut F_K \times \Aut F_L) \ar[u]^-{\Phi \otimes \Phi } \ar[r]^-{\ind} & \Rep(\Aut( F_K \sqcup F_L)) \ar[u]_-{\Phi}
}
\end{equation}
commutes. This commutativity is a special case of  \cite[Theorems 3.6 \& 3.14]{CMO1}. To be specific, set $\Zelev{G}=G_{K\sqcup L}$, $\Zelev{L}=G_{K,L}$, $\Zelev{U}=\Zelev{U_0}=U_{K,L}^{K\sqcup L}$, $\Zelev{V}=\Zelev{V_0}=\{1\}$, $\Zelev{G_0}=\Map(Y_{K\sqcup L},H)$, $\Zelev{L_0}=\Map(Y_{K,L},H)$, and $\Zelev{\psi} = \pi_{F_K}\otimes_{\k}\pi_{F_L} \in \Irr(\Map(Y_{K,L},H))$. (We are writing $\Zelev{G}$ (etc.) to designate the object called $G$ (etc.) in \cite[Section 3]{CMO1}.) We then have, still in the notation of \cite{CMO1}, $\Zelev{\phi}=\pi_{F_K\sqcup F_L}$, $\Zelev{L(\psi)}=(\Aut F_K\times \Aut F_L)\ltimes \Map(Y_{K,L},H)$, $\Zelev{G(\phi)}=\Aut(F_K \sqcup F_L)\ltimes \Map(Y_{K\sqcup L},H)$, $\Zelev{U(\phi)}=\Zelev{U}$, $\Zelev{V(\phi)}=\Zelev{V}$, $\Zelev{\overline{L}(\psi)}=\Aut F_K \times \Aut F_L$, $\Zelev{\overline{G}(\phi)}=\Aut(F_K \sqcup F_L)$, and $\Zelev{\overline{U}(\phi)}=\Zelev{\overline{V}(\phi)}=\{1\}$. The cocycles appearing in \cite[Theorem 3.14]{CMO1} are trivial in this instance. All of these identifications being made, the functor $\pind$ appearing in \cite[Theorem 3.6]{CMO1} is the functor $\pind_{K,L}^{K\sqcup L}$, while the functor $\Zelev{\pind_{\overline{U}(\phi),\overline{V}(\phi)}^{\phi'}}$ appearing in \cite[Theorem 3.14]{CMO1} is  the usual induction functor $\Rep(\Aut F_K \times \Aut F_L)\to \Rep(\Aut(F_K\sqcup F_L))$. Pasting together the two commuting diagrams from \cite[Theorems 3.6 \& 3.14]{CMO1} then yields the commuting diagram \eqref{eq:Clifford-Mackey-multiplication-diag}.
\end{proof}

To describe the comultiplication maps on $\mathcal M$ we need some more notation. 
For each $F\in \Map(Y_N,\widehat{H})$ and each subset $K\subseteq N$ we let  $F\restrict_{Y_K}$ be the restriction of the function $F$ to the subset $Y_K\subseteq Y_N$, and we let 
\[
(\Aut F)_K=\Aut_F\cap S_{K,K^c} = \{w\in \Aut F\ |\ wK=K\}
\] 
be the stabiliser of $K$ for the action of $\Aut F\subseteq S_N$ on the power set $\Power{N}$.  The group $(\Aut F)_K$ leaves the subsets $Y_K,Y_{K^c}\subseteq Y_N$ invariant, and we obtain an embedding of groups 
\[
(\Aut F)_{K} \xrightarrow{w\mapsto (w\restrict_{Y_K}, w\restrict_{Y_{K^c}})} \Aut(F\restrict_{Y_K})\times \Aut(F\restrict_{Y_{K^c}}) \subseteq S_K\times S_L.
\]

Now given $F\in \Map(Y_N,\widehat{H})$  and a representation $\gamma$ of $\Aut F$ we define
\begin{equation}\label{eq:Delta-M-definition}
\Delta_{\mathcal M}\gamma \coloneqq \sum_{\Aut F(K)\in \orbits{\Aut F}{\Power{N}}} \ind_{(\Aut F)_K}^{\Aut(F\restrict_{Y_K})\times \Aut(F\restrict_{Y_{K^c}})}\big( ( \res^{\Aut F}_{(\Aut F)_K} \gamma)\otimes_{\k}\pi_{F\restrict_{Y_N\setminus Y_{K,K^c}}}\big).
\end{equation}
We are summing over a set of representatives for the set $\orbits{\Aut F}{\Power{N}}$ of $\Aut F$-orbits in $\Power{N}$; the group $(\Aut F)_K $ acts on the representation $\pi_{F\restrict_{Y_N\setminus Y_{K,K^c}}}=\bigotimes_{y\in Y_N\setminus Y_{K,K^c}} F(y)$ by permuting the tensor factors; and each summand on the right-hand side of the formula is regarded as an element of $\mathcal M\otimes_{\Z}\mathcal M$ via the canonical isomorphisms $R(G\times G')\cong R(G)\otimes_{\Z} R(G')$. 
Note that when $H$ is abelian, so that each of its irreducible representations is one-dimensional, the representation $\pi_{F\restrict_{Y_N\setminus Y_{K,K^c}}}$ is the trivial one-dimensional representation of $(\Aut F)_K$, so the above formula simplifies to 
\begin{equation}\label{eq:Delta-M-abelian}
\Delta_{\mathcal M}\gamma \coloneqq \sum_{\Aut F(K)\in \orbits{\Aut F}{\Power{N}}} \ind_{(\Aut F)_K}^{\Aut(F\restrict_{Y_K})\times \Aut(F\restrict_{Y_{K^c}})} \res^{\Aut F}_{(\Aut F)_K} \gamma \qquad (H\text{ abelian}).
\end{equation}

To define the second comultiplication $\delta_{\mathcal M}$ we need one more piece of terminology: the \emph{support} of a function $F\in \Map(Y_N,\widehat{H})$ is defined by $\supp F = \{y\in Y_N\ |\ F(y)\neq \triv_H\}$. This is an $\Aut F$-invariant subset of $Y_N$. 

Now for $F\in \Map(Y_N,\widehat{H})$ and $\gamma\in \Rep(\Aut F)$ we define
\begin{equation}\label{eq:delta-M-definition}
\delta_{\mathcal M}\gamma \coloneqq \sum_{\substack{\Aut F(K)\in \orbits{\Aut F}{\Power{N}}\\ \supp F \subseteq Y_{K,K^c}}} \ind_{(\Aut F)_K}^{\Aut(F\restrict_{Y_K})\times \Aut(F\restrict_{Y_{K^c}})} \res^{\Aut F}_{(\Aut F)_K} \gamma.
\end{equation}

Finally, we define $\epsilon_{\mathcal M}:\mathcal M\to \Z$ by declaring that for the empty function $F_\emptyset\in \Map(Y_{\emptyset},\widehat{H})$ the map $\epsilon_{\mathcal M}:R(\Aut F_\emptyset)\to \Z$ sends the trivial representation to $1$, while for all $N\neq\emptyset$ and all $F\in \Map(Y_N,\widehat{H})$ the map $\epsilon_{\mathcal M}:R(\Aut F)\to \Z$ is identically zero.

\begin{corollary}\label{cor:M-Hopf}
The graded algebra isomorphism $\Phi:\mathcal M\to \mathcal R$ of Proposition \ref{prop:Clifford-Mackey-multiplication} relates the maps $\Delta_{\mathcal M}$, $\delta_{\mathcal M}$, and $\epsilon_{\mathcal M}$ defined above to the structure maps $\Delta$, $\delta$, and $\epsilon$ on $\mathcal R$ as follows:
\[
\Delta\Phi = (\Phi\otimes \Phi)\Delta_{\mathcal M}, \qquad \delta\Phi=(\Phi\otimes\Phi)\delta_{\mathcal M},\qquad \text{and}\qquad  \epsilon\Phi=\epsilon_{\mathcal M}.
\]
Consequently the graded algebra $\mathcal M$ equipped with the comultiplication $\Delta_{\mathcal M}$ and counit $\epsilon_{\mathcal M}$ becomes a connected, graded, commutative, and cocommutative Hopf algebra; while the graded algebra $\mathcal M$ equipped with the comultiplication $\delta_{\mathcal M}$, the counit $\epsilon_{\mathcal M}$, and the basis \eqref{eq:M-basis} becomes a PSH algebra.
\end{corollary}

\begin{proof}
The identity $\epsilon\Phi=\epsilon_{\mathcal M}$ is easily verified: $\Phi$ is an isomorphism of unital graded algebras, and $\epsilon$ and $\epsilon_{\mathcal M}$ are the inverses of the respective unit maps. 

To verify the formula for $\Delta_{\mathcal M}$ we fix a function $F\in \Map(Y_N,\widehat{H})$, a representation $\gamma$ of $\Aut F$, and a subset $K\subseteq N$. We will prove that the term
\begin{equation}\label{eq:M-Hopf-1}
(\Delta \Phi \gamma)_K \coloneqq \res^N_{K,K^c}  \ind_{G_F}^{G_N}(\gamma\ltimes \pi_F)
\end{equation}
corresponding to the orbit $S_N(K)$ in the definition of $\Delta$ is equal to the sum
\begin{equation}\label{eq:M-Hopf-2}
\begin{aligned}
& (\Phi\otimes\Phi)(\Delta_{\mathcal M}\gamma)_K \\
& = \sum_{\Aut F(L)\in \orbits{\Aut F}{S_N(K)}} (\Phi\otimes\Phi)\ind_{(\Aut F)_L}^{\Aut(F\restrict_{Y_L})\times \Aut(F\restrict_{Y_{L^c}})}\big( (\res^{\Aut F}_{(\Aut F)_L} \gamma) \otimes_{\k} \pi_{F\restrict_{Y_N\setminus Y_{L,L^c}}} \big)
\end{aligned}
\end{equation}
of the images, under $\Phi\otimes\Phi$ of the terms in the sum \eqref{eq:Delta-M-definition} associated to the $\Aut F$-orbits in $S_N(K)$.

Choose a set $W$ of representatives for the double-coset space $\Aut F \backslash S_N / S_{K,K^c}$. Observing that
\[
\begin{aligned}
G_F\backslash G_N/ G_{K,K^c} & = (\Aut F\ltimes \Map({Y_{N}},H))\backslash (S_N\ltimes \Map({Y_N},H)) / (S_{K,K^c} \ltimes \Map({Y_{K,K^c}},H)) \\ 
& \cong \Aut F\backslash S_N/ S_{K,K^c}
\end{aligned}
\]
shows that $W$ is also a set of representatives for $G_F \backslash G_N /G_{K,K^c}$. 

Now $S_{K,K^c}$ is precisely the isotropy group of $K$ for the action of $S_N$ on $\Power{N}$, and so the map $w\mapsto \Aut F(wK)$ gives a bijection $W\cong \Aut F\backslash S_N(K)$. 
Applying the standard Mackey formula \cite[Theorem 1]{Mackey} to \eqref{eq:M-Hopf-1}, using the set of double-coset representatives $W$ and recalling that the relation $\Ad_w\rho=\rho$ holds in $\mathcal R$, we find
\begin{equation}\label{eq:Mackey-coproduct-step2}
\begin{aligned}
\big(\Delta\Phi\gamma\big)_{K} & = \sum_{w\in W} \Ad_{w^{-1}} \ind_{{}^wG_{K,K^c}\cap G_F}^{{}^w G_{K,K^c}} \res^{G_F}_{{}^wG_{K,K^c}\cap G_F} (\gamma\ltimes \pi_F) \\
& = \sum_{\Aut F(L)\in \Aut F\backslash S_N(K)} \ind_{G_{L,L^c}\cap G_F}^{G_{L,L^c}} \res^{G_F}_{G_{L,L^c}\cap G_F}(\gamma\ltimes \pi_F).
\end{aligned}
\end{equation}

For each $L\in S_N(K)$ we have $G_{L,L^c}\cap G_F = (\Aut F)_L \ltimes \Map({Y_{L,L^c}},H)$. The restriction of the representation $\gamma\ltimes \pi_F$ to this group is equal to 
\[
(\res^{\Aut F}_{(\Aut F)_L} \gamma)\ltimes (\pi_{F\restrict_{Y_L}}\otimes_{\k}\pi_{F\restrict_{Y_{L^c}}}\otimes_{\k} \pi_{F\restrict_{Y_N\setminus Y_{L,L^c}}}).
\]
The group $\Map({Y_{L,L^c}},H)$ acts trivially in the representation $\pi_{F\restrict_{Y_N\setminus Y_{L,L^c}}}$, so we may rewrite this last displayed representation as
\[
\big( (\res^{\Aut F}_{(\Aut F)_L} \gamma) \otimes_{\k} \pi_{F\restrict_{Y_N\setminus Y_{L,L^c}}} \big) \ltimes (\pi_{F\restrict_{Y_L}}\otimes_{\k}\pi_{F\restrict_{Y_{L^c}}}).
\]
Temporarily writing $A_L\coloneqq \Aut(F\restrict_{Y_L})\times \Aut (F\restrict_{Y_{L^c}})$ and $B_L\coloneqq \Map({Y_{L,L^c}},H)$ to compactify the notation, we continue the computation from \eqref{eq:Mackey-coproduct-step2} to find
\begin{equation*}\label{eq:Mackey-coproduct-step3}
\begin{aligned}
& \big(\Delta\Phi\gamma\big)_{K}  = \sum_{\Aut F(L)} \ind_{(\Aut F)_L \ltimes B_L}^{G_{L,L^c}} \big(\big( (\res^{\Aut F}_{(\Aut F)_L} \gamma) \otimes_{\k} \pi_{F\restrict_{Y_N\setminus Y_{L,L^c}}} \big)\ltimes (\pi_{F\restrict_L}\otimes_{\k}\pi_{F\restrict_{L^c}}) \big) \\
& = \sum_{\Aut F(L)} \ind_{A_L\ltimes B_L}^{G_{L,L^c}} \ind_{(\Aut F)_L \ltimes B_L}^{A_L\ltimes B_L} \big(\big( (\res^{\Aut F}_{(\Aut F)_L} \gamma) \otimes_{\k} \pi_{F\restrict_{Y_N\setminus Y_{L,L^c}}} \big) \ltimes (\pi_{F\restrict_L}\otimes_{\k}\pi_{F\restrict_{L^c}}) \big) \\
& = \sum_{\Aut F(L)} \ind_{A_L\ltimes B_L}^{G_{L,L^c}} \bigg( \Big( \ind_{(\Aut F)_L}^{A_L} \big( (\res^{\Aut F}_{(\Aut F)_L} \gamma) \otimes_{\k} \pi_{F\restrict_{Y_N\setminus Y_{L,L^c}}} \big)\Big) \ltimes (\pi_{F\restrict_L}\otimes_{\k} \pi_{F\restrict_{L^c}})\bigg) \\
& = \sum_{\Aut F(L)} (\Phi\otimes\Phi) \ind_{(\Aut F)_L}^{A_L} \big( (\res^{\Aut F}_{(\Aut F)_L} \gamma) \otimes_{\k} \pi_{F\restrict_{Y_N\setminus Y_{L,L^c}}} \big) = (\Phi\otimes\Phi)(\Delta_{\mathcal M}\gamma)_K
\end{aligned}
\end{equation*}
as required.

We turn now to the relation $\delta\Phi=(\Phi\otimes\Phi)\delta_{\mathcal M}$, keeping all of the notation established so far. To obtain $(\delta\Phi\gamma)_{K}$ from $(\Delta\Phi\gamma)_{K}$ we must project the latter onto the space of $U_{K,K^c}^N$-fixed vectors. Equivalently, we must project each of the representations $\rho_L\coloneqq \ind_{G_{L,L^c}\cap G_F}^{G_{L,L^c}} \res^{G_F}_{G_{L,L^c}\cap G_F}(\gamma\ltimes \pi_F)$ occuring in the last line of \eqref{eq:Mackey-coproduct-step2} onto its subspace of $U_{L,L^c}^N$-invariants. The group $U_{L,L^c}^N$ acts on $\rho_L$ by a sum of $S_{L,L^c}$-conjugates of the irreducible representation $\pi_{F\restrict_{Y_N\setminus {Y_{L,L^c}}}}$, and so the space of $U_{L,L^c}^N$-fixed vectors will be zero if one of the factors $F(y)$ (for $y\in Y_N\setminus Y_{L,L^c}$) is nontrivial; while on the other hand this space of invariants will be all of $\rho_L$ if all of the $F(y)$ are trivial. Since $F(y)=\triv_H$ for all $y\in Y_N\setminus Y_{L,L^c}$ precisely when $\supp F\subset Y_{L,L^c}$, we obtain from \eqref{eq:Mackey-coproduct-step3}
\[
\begin{aligned}
 \big( \delta \Phi\gamma\big)_{K}  &
=  \sum_{\substack{\Aut F(L)\in \Aut F\backslash S_N(K) \\ \supp F\subseteq Y_{L,L^c}}} (\Phi\otimes\Phi)\ind_{(\Aut F)_L}^{\Aut(F\restrict_{Y_L})\times \Aut(F\restrict_{Y_{L^c}})}\big( (\res^{\Aut F}_{(\Aut F)_L} \gamma) \otimes_{\k} \pi_{F\restrict_{Y_N\setminus Y_{L,L^c}}} \big) \\
& =  \sum_{\substack{\Aut F(L)\in \Aut F\backslash S_N(K) \\ \supp F\subseteq Y_{L,L^c}}} (\Phi\otimes\Phi)\ind_{(\Aut F)_L}^{\Aut(F\restrict_{Y_L})\times \Aut(F\restrict_{Y_{L^c}})} \res^{\Aut F}_{(\Aut F)_L} \gamma \\
& = (\Phi\otimes\Phi)(\delta_{\mathcal M}\gamma)_K
\end{aligned}
\]
as required.
\end{proof}

\begin{remark}[Structure of the PSH algebra $\mathcal M^\delta_{Y,H}$]
Zelevinsky's structure theorem for PSH algebras \cite[Theorems 2.2 \& 3.1]{Zelevinsky} identifies the PSH algebra $\mathcal M^\delta_{Y,H}$ with a tensor product of copies of the Hopf algebra $\Sym_{\Z}$ of symmetric functions, indexed by the set of primitive irreducible elements of $\mathcal M^\delta_{Y,H}$. (Here \emph{irreducible elements} are, by definition, elements of the distinguished basis \eqref{eq:M-basis}.) The set of primitive irreducibles can readily be identified from the formula \eqref{eq:delta-M-definition} for the comultiplication $\delta_{\mathcal M}$. For each finite set $N\neq \emptyset$  let us call a function $F\in \Map(Y_N,\widehat{H})$ \emph{primitive} if $\supp F \not\subseteq Y_{K,K^c}$ for any $K\subsetneq N$. We denote by $\Map(Y_N,\widehat{H})_{\prim}$ the set of all such functions.
The empty function $F_\emptyset \in \Map({Y_\emptyset},\widehat{H})$ is, by definition, not primitive.

For example:
\begin{enumerate}[\rm(1)]
\item If $\#N=1$ then every function in $\Map(Y_N,\widehat{H})$ is primitive. If $Y=\id$ then these are the only primitive functions.
\item For $Y_N=N^2$, if we identify $\Map(Y_N,\widehat{H})$ with the set of $N\times N$ matrices with entries in $\widehat{H}$, then the non-primitive functions are those whose corresponding matrix can be put into block-diagonal form $\left[\begin{smallmatrix} \bigast  & \triv_H \\ \triv_H & \bigast  \end{smallmatrix}\right]$  by conjugating by a permutation matrix. 
\item For $Y_N=S_N\setminus \{\id_N\}$ a function $F\in \Map(Y_N,\widehat{H})$ is primitive if and only if its support  generates a transitive subgroup of $S_N$.
\end{enumerate}

The set of irreducible primitive elements of $\mathcal M^\delta_{Y,H}$ is now the following subset of the canonical basis:
\[
\Prim(Y,H) \coloneqq \bigg( \bigsqcup_{\substack{N\neq\emptyset\\ F\in \Map(Y_N,\widehat{H})_{\prim}}} \Irr(\Aut F) \bigg)_{\Set^\times}.
\]
As noted in \cite[4.19, 7.4]{Zelevinsky}, the structure theory of PSH algebras gives a parametrisation of the irreducible representations of the groups $G_N(Y,H)$ in terms of partition-valued functions on the set $\Prim(Y,H)$. In contrast to the cases $Y_N=\emptyset$ and $Y_N=N$ considered in \cite{Zelevinsky}, this parametrisation for a general Young set does not necessarily reduce the classification of the irreducible representations of the $G_N(Y,H)$s to a manageable problem. In the example considered in Section \ref{sec:graphs}, for instance, the set of primitive irreducibles contains all irreducible representations of all finite groups; see Remark \ref{rem:graph-PSH}.
\end{remark}

\subsection{The basic subalgebra}\label{subsec:basic}

We continue to fix a Young set $Y$ and auxiliary group $H$, and often omit them from the notation. We are going to construct Hopf subalgebras $\mathcal B^{\Delta/\delta}$ of our Hopf algebras $\mathcal M^{\Delta/\delta}\cong \mathcal R^{\Delta/\delta}$ from the representations of the  {base group} $\Map(Y_N,H)\subseteq G_N$. When $H$ is abelian the algebra $\mathcal B^{\Delta}$ is the Hopf algebra associated by Schmitt in \cite[Section 3.3]{Schmitt-HACS} to the coherent exponential species $N\mapsto \Map(Y_N,\widehat{H})$; see Proposition \ref{prop:HACS}.

As an additive group we define
\[
\mathcal B=\mathcal B_{Y,H}\coloneqq \big(\bigoplus_{N\in \Set }R(\Map(Y_N,H)) \big)_{\Set^\times} 
\]
where the subscript $\Set^\times$ again indicates coinvariants by set isomorphisms: that is, we impose the relation $\pi=\Ad_w\pi$ in $\mathcal B$ for all representations $\pi$ of $\Map(Y_N,H)$ and all bijective maps $w:N\to M$. Thus $\mathcal B$ is a free abelian group with basis 
\[
\big( \bigsqcup_{N} \{\pi_F\ |\ F\in\Map(Y_N,\widehat{H})\}\big)_{\Set^\times}.
\] 
We grade $\mathcal B$ by putting $R(\Map(Y_N,H))$ in degree $\#N$.

The operation 
\[
\Map(Y_K,\widehat{H})\times \Map(Y_L,\widehat{H}) \xrightarrow{(F_K,F_L)\mapsto F_K\sqcup F_L} \Map({Y_{K\sqcup L}},\widehat{H})
\]
of \eqref{eq:F-sqcup} induces a multiplication 
\[
R(\Map({Y_K},H))\otimes_{\Z} R(\Map({Y_L},H)) \to R(\Map(Y_{K\sqcup L},H)),
\]
turning  $\mathcal B$ into an associative graded algebra, with unit $\pi_{F_\emptyset}$ (the one-dimensional trivial representation of the trivial group $\Map({Y_\emptyset},H)$). We define the counit $\epsilon_{\mathcal B}$ by $\epsilon_{\mathcal B}\pi_{F_\emptyset}=1$ and $\epsilon_{\mathcal B}\pi_F=0$ for all other $F$.

Given $F\in \Map(Y_N,\widehat{H})$ we define
\begin{equation}\label{eq:Delta-B-def}
\begin{aligned}
\Delta_{\mathcal B} \pi_F  & \coloneqq \sum_{K\subseteq N} \left(\dim \pi_{F\restrict_{Y_N\setminus Y_{K,K^c}}}\right) \pi_{F\restrict_{Y_{K}}} \otimes \pi_{F\restrict_{Y_{K^c}}}\qquad \text{and}\\
\delta_{\mathcal B} \pi_F  & \coloneqq \sum_{\substack{K\subseteq N \\ \supp F \subseteq Y_{K,K^c}}} \pi_{F\restrict_{Y_{K}}} \otimes \pi_{F\restrict_{Y_{K^c}}}.
\end{aligned}
\end{equation}
Here we have
\[
\dim \pi_{F\restrict_{Y_N\setminus Y_{K,K^c}}} = \prod_{y\in Y_N\setminus Y_{K,K^c}} \dim  F(y)
\]
where $\dim$ denotes the dimension of the underlying $\k$-vector space. When $H$ is abelian all of these dimensions are $1$ and so the formula for $\Delta_{\mathcal B}$ simplifies to
\begin{equation}\label{eq:Delta-B-abelian}
\qquad \qquad \qquad \Delta_{\mathcal B} \pi_F = \sum_{K\subseteq N}  \pi_{F\restrict_{Y_{K}}} \otimes \pi_{F\restrict_{Y_{K^c}}} \qquad\qquad \text{($H$ abelian)}.
\end{equation}

\begin{example}\label{example:binomial-Hopf-algebra}
Taking $H=1$ the trivial group, and $Y$ an arbitrary Young set, there is a unique $F_N\in \Map(Y_N,\widehat{H})$ for each finite set $N$, and the map $\pi_{F_N}\mapsto x^{\#N}$ induces an isomorphism $\mathcal B^{\delta}_{Y,1}=\mathcal B^{\Delta}_{Y,1}\xrightarrow{\cong} \Z[x]$ to the binomial Hopf algebra over $\Z$, i.e., $\Z[x]$ with its usual multiplication and with comultiplication $\Delta(x^n) = \sum_{k=0}^n \binom{n}{k} x^k\otimes x^{n-k}$.
\end{example}

For each finite group $G$ we let $\reg_G$ be the regular representation: i.e., the representation on $\k^G$ by permuting coordinates.

\begin{proposition}\label{prop:regular-embedding}
The map
\[
\reg: \mathcal B \to \mathcal M,\qquad \pi_F \mapsto \reg_{\Aut F}
\]
is an embedding of unital graded algebras, and it satisfies
\[
\Delta_{\mathcal M} \reg = (\reg\otimes \reg)\Delta_{\mathcal B},\quad 
\delta_{\mathcal M}\reg = (\reg\otimes\reg)\delta_{\mathcal B},\quad \text{and} \quad
\epsilon_{\mathcal M}\reg = \reg \epsilon_{\mathcal B}.
\]
Thus the comultiplication maps $\Delta_{\mathcal B}$ and $\delta_{\mathcal B}$ each equip $\mathcal B$ with the structure of a connected, commutative and cocommutative graded Hopf algebra.
\end{proposition}

Note that the map $\reg:\mathcal B\to \mathcal M$ does not send irreducibles to irreducibles. In particular, $\mathcal B^\delta$ is not a PSH algebra, as is already evident in Example \ref{example:binomial-Hopf-algebra}.

\begin{proof}
The map $\reg$ is clearly injective, graded, and intertwines the units and counits. It is also easy to see that $\reg$ is a morphism of algebras: given $F_K\in \Map(Y_K,\widehat{H})$ and $F_L\in \Map(Y_L,\widehat{H})$, the tensor product $\reg_{F_K}\otimes_{\k} \reg_{F_L}$ is the regular representation of $\Aut F_K \times \Aut F_L$, and performing the multiplication in $\mathcal M$---i.e., inducing this representation  up to $\Aut(F_K\sqcup F_K)$---gives the regular representation of  $\Aut(F_K\sqcup F_K)$.

It remains to prove that $\Delta_{\mathcal M}\reg = (\reg\otimes\reg)\Delta_{\mathcal B}$ and $\delta_{\mathcal M}\reg=(\reg\otimes\reg)\delta_{\mathcal B}$. To do this we first note that for each $w\in \Aut F$ and each $K\subseteq N$ we have $\pi_{F\restrict_{Y_{wK}}}=\pi_{F\restrict_{wY_K}} = \Ad_{w} \pi_{(w^{-1}F)\restrict_{Y_K}} = \pi_{F\restrict_{Y_K}}$ in $\mathcal B$. So the summands in the definition \eqref{eq:Delta-B-def} are constant on the $\Aut F$-orbits in $\Power{N}$. The number of sets $wK$ in the orbit $\Aut F(K)$ is equal to the index $[\Aut F : (\Aut F)_K]$, and so we may rewrite the definitions as follows:
\[
\begin{aligned}
\Delta_{\mathcal B} \pi_F  & \coloneqq \sum_{\Aut F(K) \in \Aut F\backslash \Power{N}} \left(\dim \pi_{F\restrict_{Y_N\setminus Y_{K,K^c}}}\right) [\Aut F: (\Aut F)_K] \pi_{F\restrict_{Y_{K}}} \otimes \pi_{F\restrict_{Y_{K^c}}}\\
\delta_{\mathcal B} \pi_F  & \coloneqq \sum_{\substack{\Aut F(K)\in \Aut F\backslash \Power{N}  \\ \supp F \subseteq Y_{K,K^c}}} [\Aut F: (\Aut F)_K] \pi_{F\restrict_{Y_{K}}} \otimes \pi_{F\restrict_{Y_{K^c}}}.
\end{aligned}
\]

Comparing the above formulas with the definitions \eqref{eq:Delta-M-definition} and \eqref{eq:delta-M-definition} of $\Delta_{\mathcal M}$ and $\delta_{\mathcal M}$, we see that we must prove that for all $K\subseteq N$ and all  $F\in \Map(Y_N,\widehat{H})$ that 
\[
\begin{aligned}
& \ind_{(\Aut F)_K}^{\Aut(F\restrict_{Y_K})\times \Aut(F\restrict_{Y_{K^c}})}\big( ( \res^{\Aut F}_{(\Aut F)_K} \reg_{\Aut F})\otimes_{\k}\pi_{F\restrict_{Y_N\setminus Y_{K,K^c}}}\big)\\
 & = (\dim \pi_{F\restrict_{Y_N\setminus Y_{K,K^c}}}) [\Aut F : (\Aut F)_{K}] \reg_{\Aut(F\restrict_{Y_K})\times \Aut( F\restrict_{Y_{K^c}})}.
\end{aligned}
\]
Since $\ind$ sends regular representations to regular representations it will suffice to prove that 
\begin{equation}\label{eq:B-proof-last}
(\res^{\Aut F}_{(\Aut F)_{K}} \reg_{\Aut F})\otimes_{\k} \pi_{F\restrict_{Y_N\setminus Y_{K,K^c}}} = (\dim \pi_{F\restrict_{Y_N\setminus Y_{K,K^c}}}) [\Aut F : (\Aut F)_{K}] \reg_{(\Aut F)_K}.
\end{equation}
For every group $G$, subgroup $G'$, and representation $\rho\in \Rep(G')$ we have  
\[
\res^{G}_{G'}\reg_{G} = [G: G'] \reg_{G'} \quad \text{ and }\quad \reg_{G'}\otimes_{\k} \rho = (\dim \rho)\reg_{G'}
\]
so the equality \eqref{eq:B-proof-last} does hold.
\end{proof}

When $H$ is abelian the Hopf algebra $\mathcal B^{\Delta}$ is the same as one constructed in \cite{Schmitt-HACS}, as we shall now explain.

\begin{proposition}\label{prop:HACS}
Let $Y$ be a Young set and let $H$ be a finite \emph{abelian} group.
\begin{enumerate}[\rm(1)]
\item The contravariant functor $E:\Set^{\inj}\to \Set$ defined on objects by $N\mapsto \Map(Y_N,\widehat{H})$ and on morphisms by $w \mapsto (F\mapsto F\circ Y_w)$ is a \emph{coherent exponential $R$-species} as defined in \cite[3.3]{Schmitt-HACS}: it is the exponential of the contravariant functor $\Set^\times\to \Set^\times$ given by $N\mapsto \Map(Y_N,\widehat{H})_{\prim}$.
\item The Hopf algebra $\mathcal B_{Y,H}^{\Delta}$ is isomorphic to the Hopf algebra $\mathcal B_E$ associated to the coherent exponential species $E$ in \cite[3.3]{Schmitt-HACS}.
\end{enumerate}
\end{proposition}
 
 \begin{proof}
 Fix a finite set $N$ and a function $F\in \Map(Y_N,\widehat{H})$. There is a unique partition $\lambda=(L_i\ |\ i\in I)\in \Part_N$ and primitive functions $F_i\in \Map({Y_{L_i}},\widehat{H})_{\prim}$ such that $F=\bigsqcup_{i\in I} F_i$: namely, take $\lambda \coloneqq \bigwedge_{\supp F\subseteq Y_{\lambda'}} \lambda'$ and, writing $\lambda=(L_i \ |\ i\in I)$, take $F_i\coloneqq F\restrict_{L_i}$. The map sending $F$ to the assembly $\{F_i\ |\ i\in I\}$ then identifies $E$ with the exponential species $\exp \Map({Y},\widehat{H})_{\prim}$. The coherence of this species amounts to the property that for each subset $K\subseteq N$ we have $F\restrict_K = \bigsqcup_i F_i\restrict_{K\cap L_i}$, where the $L_i$ and $F_i$ are as above; this is clear, since $\supp (F\restrict_K) = K\cap \supp F \subseteq K\cap Y_\lambda$, and $F_i=F\restrict_{L_i}$. Now the identification between Schmitt's $\mathcal B_E$ and our $\mathcal B_{Y,H}^\Delta$ follows immediately from a comparison of the definitions of multiplication and comultiplication in these two Hopf algebras.
 \end{proof}

\subsection{The canonical character and symmetric functions}\label{subsec:zeta}
We conclude our general study of the Hopf algebras $\mathcal R_{Y,H}^{\Delta/\delta}$, $\mathcal M_{Y,H}^{\Delta/\delta}$, and $\mathcal B_{Y,H}^{\Delta/\delta}$ by observing  that they all carry a canonical $\Z$-valued character, and hence a canonical Hopf-algebra homomorphism into the Hopf algebra of symmetric functions. 

\begin{definition}\label{def:zeta}
Let $Y$ be a Young set and let $H$ be a finite group, and consider the algebras $\mathcal R=\mathcal R_{Y,H}$, $\mathcal M=\mathcal M_{Y,H}$, and $\mathcal B=\mathcal B_{Y,H}$. We define $\Z$-linear maps $\zeta_{\mathcal R}:\mathcal R\to \Z$, $\zeta_{\mathcal M}:\mathcal M\to \Z$ and $\zeta_{\mathcal B}:\mathcal B\to \Z$ as follows:
\begin{enumerate}[$\square$]
\item For each finite set $N$ and each $\rho\in \Irr(G_N(Y,H))$,
\[
\zeta_{\mathcal R}(\rho) = \begin{cases} 1 & \text{if }\rho=\triv_{G_N} \\ 0 & \text{otherwise.}\end{cases}
\]
\item For each $F\in \Map(Y_N,\widehat{H})$ and each $\gamma\in \Irr(\Aut F)$,
\[
\zeta_{\mathcal M}(\gamma) = \begin{cases} 1 & \text{if $\supp F=\emptyset$ and $\gamma=\triv_{\Aut F}$}\\ 0 & \text{otherwise.}
\end{cases}
\]
\item For each $F\in \Map(Y_N,\widehat{H})$,
\[
\zeta_{\mathcal B}(\pi_F) = \begin{cases} 1 & \text{if }\supp F=\emptyset \\ 0 & \text{otherwise.}\end{cases}
\]
\end{enumerate}
\end{definition}

\begin{lemma}\label{lem:zeta}
Each of the maps $\zeta$ defined above is an algebra homomorphism, and the diagram
\[
\xymatrix@C=60pt@R=30pt{
\mathcal B  \ar[r]^-{\reg} \ar[dr]_-{\zeta_{\mathcal B }} & \mathcal M  \ar[d]^-{\zeta_{\mathcal M }} \ar[r]^-{\Phi}_-{\cong} & \mathcal R \ar[dl]^-{\zeta_{\mathcal R }}\\
& \Z &
}
\]
commutes. Here $\Phi$ is the isomorphism of Proposition \ref{prop:Clifford-Mackey-multiplication}.
\end{lemma}

\begin{proof}
We will prove that the diagram commutes and that $\zeta_{\mathcal R}$ is an algebra homomorphism. Since the maps $\mathcal B\to \mathcal M \to \mathcal R$ are algebra homomorphisms, this will imply that $\zeta_{\mathcal M}$ and $\zeta_{\mathcal B}$ are also algebra homomorphisms.

For each $F\in \Map(Y_N,\widehat{H})$ the regular representation of $\Aut F$ decomposes as one copy of the trivial representation plus some nontrivial representations. We thus have 
\[
\zeta_{\mathcal M}\reg(\pi_F) = \zeta_{\mathcal M}(\triv_{\Aut F}) + \sum \zeta_{\mathcal M}
{\small\left(\begin{array}{c}\text{nontrivial}\\ \text{representations} \end{array} \right) }= \begin{cases} 1 & \text{if $\supp F=\emptyset$} \\ 0 & \text{otherwise}\end{cases}
\]
which is equal to $\zeta_{\mathcal B}(\pi_F)$. So the left-hand triangle in the diagram commutes.

Next, given $F\in \Map(Y_N,\widehat{H})$ and $\gamma\in \Irr(\Aut F)$, recall that the isomorphism $\Phi:\mathcal M\to \mathcal R$ of Proposition \ref{prop:Clifford-Mackey-multiplication} sends $\gamma\in \Irr(\Aut F)$ to the representation $\ind_{G_F}^{G_N} (\gamma\ltimes \pi_F)$ of $G_N$. This representation is trivial precisely when $\pi_F$ is the trivial representation of $\Map(Y_N,H)$---i.e., when $\supp F=\emptyset$---and when $\gamma$ is the trivial representation of $\Aut F=S_N$. Thus $\zeta_{\mathcal R} \Phi(\gamma) = \zeta_{\mathcal M}(\gamma)$, and so the diagram in the lemma commutes.

Finally, to show that $\zeta_{\mathcal R}$ is an algebra homomorphism, fix finite sets $K$ and $L$ and irreducible representations $\rho_K\in \Irr(G_K)$ and $\rho_L\in \Irr(G_L)$. The product $\rho_K\rho_L$ of these representations in $\mathcal R$ is the representation $\pind_{K,L}^{K\sqcup L}(\rho_K\otimes_{\k}\rho_L)$ of $G_{K\sqcup L}$. Since $\pind_{K,L}^N$ is adjoint to $\pres_{K,L}^N$ (Lemma \ref{lem:adjoints}), and since $\pres^N_{K,L}(\triv_{G_N}) = \triv_{G_{K,L}}$ (obviously), we have
\[
\begin{aligned}
\zeta_{\mathcal R}(\rho_K\rho_L)  & = \dim \Hom_{G_N} \left(\triv_{G_N}, \pind_{K,L}^N(\rho_K\otimes_{\k}\rho_L)\right) \\
&  = \dim \Hom_{G_{K,L}} \left(\triv_{G_{K,L}}, \rho_L\otimes_{\k} \rho_L \right) .
\end{aligned}
\]
The last intertwining space is one-dimensional if both $\rho_K$ and $\rho_L$ are trivial, and it is zero otherwise. Thus $\zeta_{\mathcal R}(\rho_K\rho_L)=\zeta_{\mathcal R}(\rho_K)\zeta_{\mathcal R}(\rho_L)$ as required.
\end{proof}

Let $\Sym_{\Z}$ denote the Hopf algebra of symmetric functions, in variables $x_1,x_2,\ldots$, with $\Z$ coefficients. It follows from \cite[Theorem 4.3]{ABS} and Lemma \ref{lem:zeta} that there is a commuting diagram of morphisms of Hopf algebras
\begin{equation}\label{eq:Psi-diagram}
\xymatrix@C=60pt@R=30pt{
\mathcal B^{\Delta}  \ar[r]^-{\reg} \ar[dr]_-{\Psi_{\mathcal B}} & \mathcal M^\Delta \ar[d]^-{\Psi_{\mathcal M}} \ar[r]^-{\Phi}_-{\cong} & \mathcal R^\Delta\ar[dl]^-{\Psi_{\mathcal R }}\\
& \Sym_{\Z} &
}
\end{equation}
determined uniquely by the requirement that 
\[
\Psi_{\mathcal R}(\rho)(1,0,0,\ldots)  = \zeta_{\mathcal R}(\rho)\quad \text{for all }\rho\in \mathcal R.
\]
(There is also a corresponding diagram for $\mathcal B^\delta$, $\mathcal M^\delta$, and $\mathcal R^\delta$, but here we shall focus on the $\Delta$ Hopf algebras.) We are going to compute the maps $\Psi$ explicitly in terms of monomial symmetric functions, assuming the auxiliary group $H$ to be abelian. First we shall need some more notation.

Consider the set  $\Comp_N$ of \emph{compositions} of $N$: these are \emph{ordered} lists $\kappa=(K_1,\ldots, K_{\ell})$ of mutually disjoint, nonempty blocks $K_i\subseteq N$ satisfying $\bigcup_i K_i=N$. Each composition $\kappa$ determines a partition  $\overline{\kappa}\in \Part_N$ by forgetting the order of the blocks, and we shall accordingly extend the notation previously established for partitions to compositions: thus $G_\kappa$ means $G_{\ol\kappa}$, and so on. 

As with partitions, the group $S_N$ acts on $\Comp_N$. The isotropy group of $\kappa\in \Comp_N$ is precisely the Young subgroup $S_{\kappa}\subseteq S_N$.  The $S_N$-orbits in $\Comp_N$ are parametrised by the set of \emph{integer compositions} $\Comp_{\# N}$---i.e., the set of ordered lists of positive integers summing to $\#N$---via the map sending a set composition $\kappa=(K_1,\ldots,K_\ell)$ to the integer composition $\#\kappa=(\#K_1,\ldots,\#K_\ell)$. For each integer composition $\alpha \in \Comp_{\#N}$ we let $\Comp_{N,\alpha}$ denote the corresponding orbit:
\[
\Comp_{N,\alpha}\coloneqq \{\kappa\in \Comp_N\ |\ \#\kappa=\alpha\}.
\]
Now for each function $F\in \Map(Y_N,\widehat{H})$, and each integer composition $\alpha \in \Comp_{\#N}$, we define
\[
\Comp_{F, \alpha} \coloneqq \{\kappa \in \Comp_{N,\alpha}\ |\ \supp F\cap Y_{{\kappa}}=\emptyset\}.
\]
The action of $S_N$ on $\Comp_N$ restricts to an action of $\Aut F$ on $\Comp_{F, \alpha}$, and we let 
\[
\rho_{F,\alpha}: \Aut F \to \GL(\k^{\Comp_{F,\alpha}})
\]
be the corresponding permutation representation.

For each integer composition $\alpha\in \Comp_n$ we let $M_\alpha$ denote the associated monomial quasisymmetric function \cite[7.19]{Stanley-EC2}. Although we will ultimately be writing down formulas for symmetric functions, the formulas are more natural when written in terms of the quasisymmetric functions $M_{\alpha}$, and so we shall present our functions in that form.

We now return  to the diagram \eqref{eq:Psi-diagram}. An explicit formula, involving iterated comultiplication, is given in \cite[(4.2)]{ABS} for the Hopf-algebra morphism to $\Sym_{\Z}$ induced by a character. In the case of the Hopf algebra $\mathcal R^\Delta$, where the comultiplication $\Delta$ is given by the simple formula \eqref{eq:Delta-R-definition},  the iterates of the comultiplication are easily computed, and the formula \cite[(4.2)]{ABS} for the map $\Psi_{\mathcal R}:\mathcal R\to \Sym_{\Z}$ takes a correspondingly simple form: for each representation $\rho$ of $G_N$ we have
\begin{equation}\label{eq:Psi-R}
\Psi_{\mathcal R}(\rho) = \sum_{\alpha\in \Comp_{\#N}} \left(\dim \rho^{G_{\kappa_\alpha}}\right) M_{\alpha}
\end{equation}
where $\kappa_\alpha$ is any element of the orbit $\Comp_{N,\alpha}$, and $\dim \rho^{G_{\kappa_{\alpha}}}$ is the $\k$-dimension of the space of $G_{\kappa_\alpha}$-fixed vectors in the representation $\rho$. The formula \eqref{eq:Psi-R} is valid for all auxiliary groups $H$. We shall now use this formula and the diagram \eqref{eq:Psi-diagram}  to compute the maps $\Psi_{\mathcal M}$ and $\Psi_{\mathcal B}$, under the assumption that $H$ is abelian.

\begin{proposition}\label{prop:Psi-M}
Let $Y$ be a Young set, let $H$ be a finite \emph{abelian} group, and consider the Hopf algebras $\mathcal M=\mathcal M^{\Delta}_{Y,H}$ and $\mathcal B=\mathcal B_{Y,H}^{\Delta}$. The Hopf-algebra morphism $\Psi_{\mathcal M}:\mathcal M\to \Sym_{\Z}$ induced by the character $\zeta_{\mathcal M}$ is given, for each finite set $N$, each function $F\in \Map(Y_N,\widehat{H})$, and each representation $\gamma$ of $\Aut F$, by
\[
\Psi_{\mathcal M}(\gamma) =  \sum_{\alpha \in \Comp_{\#N}} \left(\dim\Hom_{\Aut F}(\rho_{F,\alpha},\gamma)\right) M_\alpha.
\]
The Hopf-algebra morphism $\Psi_{\mathcal B}:\mathcal B\to \Sym_{\Z}$ is given, for each $F\in \Map(Y_N,\widehat{H})$, by
\[
\Psi_{\mathcal B}(\pi_F) = \sum_{\alpha\in \Comp_{\#N}} \left( \# \Comp_{F,\alpha} \right) M_\alpha.
\]
\end{proposition}

\begin{proof}
Fix $N$, $F$, and $\gamma$,  let $\alpha\in \Comp_{\#N}$ be an integer composition, and let $\kappa_\alpha$ be any element of $\Comp_{N,\alpha}$.  Since $\Psi_{\mathcal M} = \Psi_{\mathcal R}\circ \Phi$ the formula \eqref{eq:Psi-R} shows that the coefficient of $M_\alpha$ in $\Psi_{\mathcal M}(\gamma)$ is the dimension of $ (\ind_{G_F}^{G_N}(\gamma\ltimes \pi_F))^{G_{\kappa_\alpha}}$.
The Mackey formula for $\ind$ and $\res$ implies that this dimension is equal to 
\begin{equation}\label{eq:Phi-M-proof-1}
\sum_{G_FwG_{\kappa_\alpha}\in G_F\backslash G_N/ G_{\kappa_\alpha}} \dim\Hom_{ {}^wG_{{\kappa_\alpha}}\cap G_F} (\triv, \gamma\ltimes \pi_F).
\end{equation}
Considering the double-coset space indexing this last sum, we find
\[
\begin{aligned}
G_F\backslash G_N/G_{\kappa_\alpha} & = (\Aut F\ltimes \Map(Y_N,H))\backslash (S_N\ltimes \Map(Y_N,H))/ (S_{\kappa_\alpha}\ltimes \Map({Y_{\kappa_\alpha}},H))  \\ & \cong \Aut F\backslash S_N / S_{\kappa_\alpha}.
\end{aligned}
\]
Now recall that $S_{\kappa_\alpha}$ is the isotropy group and $\Comp_{N,\alpha}$ is the orbit of $\kappa_\alpha$ for the action of $S_N$ on $\Comp_N$; thus the map $w\mapsto w\kappa_\alpha$ induces a bijection
\[
\Aut F\backslash S_N/S_{\kappa_\alpha} \xrightarrow{\cong} \Aut F\backslash \Comp_{N,\alpha}.
\]
For each $w\in S_N$, setting $\kappa=w\kappa_\alpha$, we have 
\[
{}^wG_{\kappa_\alpha}\cap G_F = (S_{\kappa} \cap \Aut F) \ltimes \Map({Y_{\kappa}},H) = (\Aut F)_{\kappa}\ltimes \Map({Y_{\kappa}},H)
\]
where $(\Aut F)_{\kappa}$ indicates the isotropy group of $\kappa$ in $\Aut F$.  Making these identifications, \eqref{eq:Phi-M-proof-1} becomes
\begin{equation}\label{eq:Phi-M-proof-2}
\sum_{\Aut F(\kappa)\in \Aut F\setminus \Comp_{N,\alpha}} \dim\Hom_{ (\Aut F)_{\kappa}\ltimes \Map({Y_{\kappa}},H)}(\triv,\gamma\ltimes \pi_F).
\end{equation}

For each $\kappa\in \Comp_{N,\alpha}$ the abelian group $\Map(Y_{\kappa},H)$ acts on the representation $\gamma\ltimes \pi_F$ by the character $\pi_{F\restrict_{Y_\kappa}}$. So the representation $\gamma\ltimes \pi_F$ is either trivial on $\Map(Y_{\kappa},H)$, or else it contains no nonzero $\Map(Y_{\kappa},H)$-fixed vectors. The former possibility occurs precisely when $\supp F \cap Y_{\kappa} =\emptyset$; recall that this is, by definition, the condition that $\kappa$ belong to $\Comp_{F,\alpha}$. Now the character $\pi_F$ is trivial on $\Aut F$, and so \eqref{eq:Phi-M-proof-2} is equal to 
\begin{equation*}\label{eq:Phi-M-proof-3}
\begin{aligned}
& \sum_{\Aut F(\kappa)\in \orbits{\Aut F}{\Comp_{F,\alpha}}} \dim\Hom_{(\Aut F)_{\kappa}}(\triv,\gamma) \\
& = \sum_{\Aut F(\kappa)\in \orbits{\Aut F}{ \Comp_{F,\alpha}}} \dim \Hom_{\Aut F}\left(\ind_{(\Aut F)_{\kappa}}^{\Aut F}\triv_{(\Aut F)_{\kappa}},\gamma\right)
\end{aligned}
\end{equation*}
where the equality is Frobenius reciprocity. The $\Aut F$-representation $\ind_{(\Aut F)_{\kappa}}^{\Aut F}\triv_{(\Aut F)_{\kappa}}$ is the permutation representation associated to the orbit $\Aut F(\kappa)\subseteq \Comp_{F,\alpha}$, and so summing over these orbits gives the permutation representation $\rho_{F,\alpha}$ as claimed.

Turning to $\Psi_{\mathcal B}$, using the formula for $\Psi_{\mathcal M}$  just established and the equality $\Psi_{\mathcal B}=\Psi_{\mathcal M}\circ \reg$, we find that for $F\in \Map(Y_N,\widehat{H})$ and $\alpha\in \Comp_N$ the coefficient of $M_\alpha$ in $\Psi_{\mathcal B}(\pi_F)$ is
\[
\dim\Hom_{\Aut F}(\rho_{F,\alpha}, \reg_F) = \dim \rho_{F,\lambda} = \#\Comp_{F,\alpha}.\qedhere
\]
\end{proof}

\section{Graph automorphisms and colourings}\label{sec:graphs}

For certain choices of Young set $Y$ and auxiliary group $H$ the Hopf algebras $\mathcal M^{\delta,\Delta}_{Y,H}$ and $\mathcal B^{\delta,\Delta}_{Y,H}$, and the associated symmetric functions, admit descriptions in terms of isomorphism classes, automorphism groups, and colourings of familiar combinatorial objects. In this section we shall examine one such example.

\subsection{The Hopf algebra of graphs and chromatic symmetric functions} 

Our graphs are finite, simple, and undirected: so a graph $\Gamma$ is a finite   set $V(\Gamma)$ of vertices, and a finite set $E(\Gamma)\subseteq \{ \text{$2$-element subsets of $V(\Gamma)$}\}$ of edges. An isomorphism of graphs $\Gamma\to \Lambda$ is a bijection of vertex-sets $V(\Gamma)\to V(\Lambda)$ whose induced map on the power sets $\Power{V(\Gamma)}\to \Power{V(\Lambda)}$ restricts to a bijection $E(\Gamma)\to E(\Lambda)$. 
The disjoint union  of graphs  is defined by taking disjoint unions of vertex- and edge-sets. For each graph $\Gamma$ and each subset $U\subseteq V(\Gamma)$ the induced graph $\Gamma\restrict_U$ is defined by $V(\Gamma\restrict_U)=U$ and $E(\Gamma\restrict_U)=E(\Gamma)\cap \Power{U}$. 

Given a graph $\Gamma$ and an integer composition $\alpha=(\alpha_1,\ldots,\alpha_\ell)\in \Comp_{\#V(\Gamma)}$,   a \emph{(proper) $\alpha$-colouring} of $\Gamma$ is a function $\kappa:V(\Gamma) \to \{1,\ldots,\ell\}$ satisfying $\#\kappa^{-1}(i)=\alpha_i$ for all $i$, and $\kappa(v)\neq\kappa(w)$ for all $\{v,w\}\in E(\Gamma)$. The set of all such colourings is denoted by $\Colour_{\Gamma,\alpha}$. 

The \emph{Hopf algebra of graphs} (\cite[Section 12]{Schmitt-IHA}) is
\[
\Chromatic = \bigoplus_{[\Gamma]} \Z[\Gamma]
\]
where the sum is over the set of isomorphism classes $[\Gamma]$ of finite graphs. We grade $\Chromatic$ so that $[\Gamma]$ sits in degree $\#V(\Gamma)$. The multiplication in $\Chromatic$ is $[\Gamma]\otimes_{\Z} [\Lambda] \mapsto [\Gamma\sqcup\Lambda]$, and the comultiplication is
\[
\Delta_{\Chromatic}[\Gamma] = \sum_{U\in \Power{V(\Gamma)}} [\Gamma\restrict_U]\otimes_{\Z} [\Gamma\restrict_{U^c}]
\]
where $U^c=V(\Gamma)\setminus U$. The unit of $\Chromatic$ is the empty graph, and the counit is the map $\Chromatic\to \Z$ sending the empty graph to $1$ and all other graphs to zero. These operations make $\Chromatic$ a connected, commutative and cocommutative Hopf algebra.

The algebra $\Chromatic$ has a canonical character $\zeta_{\Chromatic}:\Chromatic\to \Z$, given by
\[
\zeta_{\Chromatic}[\Gamma]=\begin{cases} 1 & \text{if }E(\Gamma)=\emptyset \\ 0 & \text{otherwise.}\end{cases}
\]
The associated Hopf morphism $\Psi_{\Chromatic}:\Chromatic\to \Sym_{\Z}$ sends $[\Gamma]$ to the \emph{chromatic symmetric function}
\[
X_\Gamma \coloneqq \sum_{\alpha\in \Comp_{\#V(\Gamma)}} (\#\Colour_{\Gamma,\alpha}) M_\alpha.
\]
This symmetric function was first defined by Stanley in \cite{Stanley-chromatic}. The connection to the Hopf algebra $\Chromatic$ and the character $\zeta_{\Chromatic}$ was pointed out in \cite[Example 4.5]{ABS}.

\subsection{A Hopf algebra and symmetric functions associated to representations of graph automorphisms}

We are going to study an enlargement of $\Chromatic$. For each graph $\Gamma$ we let $\Aut\Gamma$ denote the group of graph-automorphisms of $\Gamma$. For each isomorphism of graphs $w:\Gamma\to \Lambda$ we have an isomorphism of Grothendieck groups
\begin{equation}\label{eq:graph-isomorphism}
\Ad_w :R(\Aut\Gamma) \xrightarrow{\ \gamma\mapsto \gamma(w^{-1}\cdot w)\ } R(\Aut\Lambda).
\end{equation} 

Consider the free abelian group
\[
\GraphAlg \coloneqq \Big(\bigoplus_{\Gamma} R(\Aut\Gamma)\Big)_{\mathrm{Graph}^\times} 
\]
where the sum is over finite simple graphs $\Gamma$, and the subscript indicates that we impose the relation $\gamma=\Ad_w\gamma$ for all $\gamma$ and $w$ as in \eqref{eq:graph-isomorphism} (that is, we take the coinvariants of the groupoid of graph isomorphisms). We grade $\GraphAlg$ so that $R(\Aut\Gamma)$ sits in degree $\#V(\Gamma)$.

Given graphs $\Gamma$ and $\Lambda$ there is an obvious inclusion of groups
\[
\Aut\Gamma \times \Aut \Lambda \into \Aut(\Gamma \sqcup \Lambda)
\]
whence an induction functor
\begin{equation*}\label{eq:induction-graph-automorphisms}
\ind_{\Aut\Gamma\times \Aut\Lambda}^{\Aut(\Gamma\sqcup\Lambda)}:\Rep(\Aut\Gamma \times \Aut\Lambda) \to \Rep(\Aut(\Gamma \sqcup \Lambda)).
\end{equation*}
On the free abelian group $\GraphAlg$ we define a graded multiplication $\GraphAlg\otimes_{\Z} \GraphAlg\to \GraphAlg$ as the direct sum of the maps
\begin{equation}\label{eq:graph-algebra-product}
R(\Aut\Gamma)\otimes_{\Z} R(\Aut\Lambda)\xrightarrow{\cong} R(\Aut\Gamma \times \Aut\Lambda) \xrightarrow{\ind} R(\Aut(\Gamma \sqcup\Lambda)).
\end{equation}
This product makes $\GraphAlg$ into an associative graded algebra; the unit is the trivial representation of the automorphism group of the empty graph.

Next, for each graph $\Gamma$ and each subset $U\subseteq V(\Gamma)$ we define
\[
(\Aut\Gamma)_U\coloneqq \{w\in \Aut\Gamma\ |\ w(U)=U\}
\]
to be the isotropy group of $U$ for the action of $\Aut\Gamma$ on the power set $\Power{V(\Gamma)}$. This is by definition a subgroup of $\Aut\Gamma$; it is also in a natural way a subgroup of the product $\Aut(\Gamma\restrict_U)\times \Aut(\Gamma\restrict_{U^c})$, via the map $w\mapsto (w\restrict_U,w\restrict_{U^c})$.

Let $\Delta_{\GraphAlg}:\GraphAlg\to \GraphAlg\otimes_{\Z}\GraphAlg$ be the graded $\Z$-linear map defined, for each graph $\Gamma$ and each representation $\gamma$ of $\Aut\Gamma$, by
\begin{equation}\label{eq:Delta-Graph-definition}
\Delta_{\GraphAlg}\gamma = \sum_{\Aut\Gamma(U)\in \orbits{\Aut\Gamma}{ \Power{V(\Gamma)}} }\ind_{(\Aut\Gamma)_U}^{\Aut(\Gamma\restrict_U)\times \Aut(\Gamma\restrict_{U^c})} \res^{\Aut\Gamma}_{(\Aut\Gamma)_U} \gamma.
\end{equation}
Here the sum is over the $\Aut\Gamma$-orbits of subsets of $V(\Gamma)$, and we are using the canonical isomorphisms $R(G)\otimes_{\Z}R(G')\xrightarrow{\cong} R(G\times G')$ to view each summand as an element of $R(\Aut(\Gamma\restrict_U))\otimes_{\Z} R(\Aut(\Gamma\restrict_{U^c}))\subset\GraphAlg\otimes_{\Z}\GraphAlg$.

We also let $\epsilon_{\GraphAlg}:\GraphAlg\to \Z$ be the map sending the trivial representation of $\Aut\emptyset$ to $1$, and all other irreducible representations to $0$. 

Let $\zeta_{\GraphAlg}:\GraphAlg\to \Z$ be the $\Z$-linear map defined, for each graph $\Gamma$ and each representation $\gamma$ of $\Aut\Gamma$, by
\[
\zeta_{\GraphAlg}(\gamma) = \begin{cases} 1 & \text{if $E(\Gamma)=\emptyset$ and $\gamma=\triv_{\Aut\Gamma}$}\\ 0 & \text{otherwise.}\end{cases}
\]

Finally, for each graph $\Gamma$ and each integer composition $\alpha\in \Comp_{\#V(\Gamma)}$ recall that $\Colour_{\Gamma,\alpha}$ is the set of proper $\alpha$-colourings of $\Gamma$. The group $\Aut\Gamma$ acts on this set by $w\kappa(v)\coloneqq \kappa(w^{-1}v)$, and we let 
\[
\rho_{\Gamma,\alpha}:\Aut \Gamma\to \GL(\k^{\Colour_{\Gamma,\alpha}})
\]
be the corresponding permutation representation.

\begin{theorem}\label{thm:GraphAlg}
\begin{enumerate}[\rm(1)]
\item The comultiplication $\Delta_{\GraphAlg}$ and counit $\epsilon_{\GraphAlg}$ make $\GraphAlg$ into a connected, commutative, and cocommutative graded Hopf algebra.
\item The map $\reg:\Chromatic\to \GraphAlg$ sending $[\Gamma]$ to the regular representation of $\Aut\Gamma$  is an embedding of Hopf algebras.
\item The map $\zeta_{\GraphAlg}:\GraphAlg \to \Z$ is an algebra homomorphism, and the induced Hopf-algebra homomorphism $\GraphAlg\to \Sym_{\Z}$ sends each representation $\gamma$ of $\Aut\Gamma$ to the symmetric function
\[
X_{\Gamma,\gamma} \coloneqq \sum_{\alpha\in \Comp_{\#V(\Gamma)}} \left( \dim\Hom_{\Aut\Gamma}(\rho_{\Gamma,\alpha},\gamma)\right) M_\alpha.
\]
\item For each graph $\Gamma$ we have $X_{\Gamma,\reg_{\Aut\Gamma}}=X_{\Gamma}$, Stanley's chromatic symmetric function; and 
\[
X_\Gamma = \sum_{\gamma\in \Irr(\Aut\Gamma)} (\dim \gamma)X_{\Gamma,\gamma}
\]
where $\dim\gamma$ is the dimension of the $\k$-vector space underlying the representation $\gamma$.
\end{enumerate}
\end{theorem}

\begin{proof}
Let  $E$ be the Young set with $E_N=\{\text{two-element subsets of $N$}\}$  (cf.~Examples \ref{examples:Young-set} and \ref{example:non-isomorphic}). We are going to prove the theorem by identifying  $\GraphAlg$ and $\Chromatic$ with the Hopf algebras $\mathcal M_{E,S_2}^{\Delta}$ and $\mathcal B_{E,S_2}^{\Delta}$ of Section \ref{sec:Hopf}.

We have $\widehat{S_2}=\{\triv,\sign\}$, so a function $F\in \Map(E_N,\widehat{S_2})$ is completely determined by its support, $\supp F = F^{-1}(\sign)$. The map sending a function $F\in \Map(E_N, \widehat{S_2})$ to the graph $\Gamma_F$ with $V(\Gamma_F)=N$ and $E(\Gamma_F)=\supp F$ is a natural isomorphism between $\Map(E_N, \widehat{S_2})$ and the set of graphs with vertex-set $N$, where `natural' means with respect to set bijections $N\to M$. In particular we have the equality $\Aut F=\Aut \Gamma_F$ of subgroups of $S_N$.

These identifications yield grading-preserving bijections between the canonical bases of $\mathcal M_{E,S_2}$ and of $\GraphAlg$, and between the canonical bases of $\mathcal B_{E,S_2}$ and of $\Chromatic$. We thus have isomorphisms of graded abelian groups
\begin{equation}\label{eq:M-G-iso}
\mathcal M_{E,S_2} \xrightarrow[\cong]{\ R(\Aut F)\ni \gamma\mapsto \gamma\in R(\Aut \Gamma_F)\ } \GraphAlg\quad \text{and}\quad \mathcal B_{E,S_2} \xrightarrow[\cong]{\pi_F\mapsto [\Gamma_F]} \Chromatic
\end{equation}
making the diagram
\[
\xymatrix{
\mathcal B_{E,S_2} \ar[r]^-{\reg} \ar[d]_-{\cong} & \mathcal M_{E,S_2} \ar[d]^-{\cong} \\
\Chromatic \ar[r]^-{\reg} & \GraphAlg
}
\]
commute. It is now an easy matter to match up the definitions of the Hopf-algebra structures and conclude that the isomorphisms \eqref{eq:M-G-iso} intertwine the units, the counits, the multiplications, the comultiplications, and the characters $\zeta$ on either side. 

To prove the formula for $X_{\Gamma,\gamma}$ in part (3) it suffices to note that for each finite set $N$, each function $F\in \Map(E_N, \widehat{S_2})$, and each integer composition $\alpha\in \Comp_{\#N}$, the map $\Comp_{F,\alpha} \to \Colour_{\Gamma_F,\alpha}$ sending the composition $(K_1,\ldots, K_\ell)$ to the proper $\alpha$-colouring $K_i\to \{i\}$ is a bijection that is equivariant for the action of $\Aut F=\Aut\Gamma_F$. Thus the given formula for $X_{\Gamma,\gamma}$ is an instance of Proposition \ref{prop:Psi-M}.

Finally, for part (4), recall that the chromatic symmetric function $X_\Gamma$ is the image of $[\Gamma]$ under the Hopf-algebra morphism $\Chromatic\to \Sym_{\Z}$ induced by the character $\zeta_\Chromatic$, while $X_{\Gamma,\reg_{\Aut\Gamma}}$ is the image of $[\Gamma]$ under the morphism induced by the character $\zeta_{\GraphAlg}\circ \reg$. Since $\zeta_{\Chromatic}= \zeta_{\GraphAlg}\circ \reg$ these two symmetric functions coincide. Now the map $\gamma\mapsto X_{\Gamma,\gamma}$ is additive, and so the asserted decomposition of $X_\Gamma$ follows from the decomposition of the regular representation of ${\Aut\Gamma}$ into irreducibles. 
\end{proof}

\begin{remark}\label{rem:graph-PSH}
We are concentrating in this section on the $\Delta$ Hopf algebras, but much of the above applies equally well to the $\delta$ Hopf algebras. The coproducts $\delta_{\Chromatic}$ and $\delta_{\GraphAlg}$ are defined by restricting the sums in the definitions of $\Delta_{\Chromatic}$ and $\Delta_{\GraphAlg}$ to subsets $U\subseteq V(\Gamma)$ that are unions of connected components of $\Gamma$. The resulting PSH algebra $\GraphAlg^\delta$ has for its set of primitive irreducible elements the union of the sets $\Irr (\Aut\Gamma)$ over the set of isomorphism classes of \emph{connected} graphs $\Gamma$. Note that every finite group arises as the automorphism group of a connected graph (\cite{Frucht} again), so the set of primitive irreducibles is still extremely complicated. Note too that the only primitive irreducible element of the Hopf algebra $\GraphAlg^\Delta$ is the trivial representation of the automorphism group of the graph with one vertex.
\end{remark}

We conclude this paper some preliminary investigations of the symmetric functions $X_{\Gamma,\gamma}$; these functions will be treated in greater detail elsewhere. We first record the formal definition of the polynomials $\chi_{\Gamma,\gamma}$ discussed in the introduction:

\begin{definition}
For each graph $\Gamma$ and each finite-dimensional representation $\gamma$ of $\Aut\Gamma$, we denote by $\chi_{\Gamma,\gamma}\in \Z[x]$ the polynomial obtained by specialisation from the symmetric function $X_{\Gamma,\gamma}$:
\[
\chi_{\Gamma,\gamma}(m) = X_{\Gamma,\gamma}(1^m) = \dim\Hom_{\Aut\Gamma}( \k^{\Colour_{\Gamma,m}},\gamma)
\]
where we recall that $\Colour_{\Gamma,m}$ is the set of proper colourings $\kappa:V(\Gamma)\to [m]$.
\end{definition}

\begin{example}
An application of Frobenius reciprocity yields the following alternative formula for the symmetric function $X_{\Gamma,\gamma}$:
\[
X_{\Gamma,\gamma} = \sum_{\substack{\alpha \in \Comp_{\#V(\Gamma)} \\ \Aut\Gamma(\kappa)\in \orbits{\Aut\Gamma}{\Colour_{\Gamma,\alpha}}}} \dim( \gamma^{(\Aut\Gamma)_{\kappa}}) M_\alpha,
\]
where $(\Aut\Gamma)_\kappa$ is the isotropy group and $\Aut\Gamma(\kappa)$ is the orbit of the colouring $\kappa$ for the action of $\Aut\Gamma$ on $\Colour_{\Gamma,\alpha}$; and $\gamma^{(\Aut\Gamma)_\kappa}$ is the space of $(\Aut\Gamma)_\kappa$-fixed vectors in the representation $\gamma$. Putting $\gamma=\triv$, the trivial representation of $\Aut \Gamma$, we obtain
\begin{equation*}\label{eq:XGammatriv}
X_{\Gamma,\triv} = \sum_{\alpha\in \Comp_{\#V(\Gamma)}} \# (\orbits{\Aut \Gamma}{\Colour_{\Gamma,\alpha}}) M_\alpha.
\end{equation*}
Specialising gives 
\[
\chi_{\Gamma,\triv}(m) = X_{\Gamma,\triv}(1^m) = \# (\orbits{\Aut\Gamma}{\Colour_{\Gamma,m}}),
\]
the orbital chromatic polynomial of $\Gamma$ \cite{Hanlon}, \cite{Cameron-Jackson-Rudd},  \cite{Cameron-Kayibi}.
\end{example}

\begin{example}
For each $n\geq 0$ let $K_n$ be the complete graph on $n$ vertices: that is, the graph in which each pair of vertices is joined by an edge. The only composition $\alpha$ of $n$ with $\Colour_{\Gamma,\alpha}\neq\emptyset$ is the composition $(1,\ldots,1)$, and so we have 
\[
X_{K_n} = n! M_{(1,\ldots,1)}.
\]
The group $\Aut K_n=S_n$ acts simply transitively on the set $\Colour_{K_n,(1,\ldots,1)}$, and so the permutation representation $\rho_{K_n,(1,\ldots,1)}$ is isomorphic to the regular representation. We thus have for each $\gamma\in \Irr(S_n)$ the equality
\[
X_{K_n, \gamma} = (\dim\gamma) M_{(1,\ldots,1)}.
\]
The decomposition of $X_{K_n}$ given in part (4) of Theorem \ref{thm:GraphAlg} thus boils down to the standard fact that the order of the group $\Aut K_n$ equals the sum of squares of the degrees of its irreducible representations.
\end{example}

This last example is typical of graphs without symmetric colourings, as the following proposition shows.

\begin{proposition}\label{prop:free-action}
For each graph $\Gamma$, the following are equivalent.
\begin{enumerate}[\rm(a)]
\item The action of $\Aut\Gamma$ on the set $\Colour_{\Gamma}$ of all proper colourings of $\Gamma$ is free. (That is, no colouring is fixed by any nontrivial automorphism.)
\item $X_{\Gamma}=(\#\Aut\Gamma) X_{\Gamma,\triv}$.
\item $\chi_\Gamma = (\#\Aut\Gamma) \chi_{\Gamma,\triv}$.
\item $X_{\Gamma,\gamma}=(\dim\gamma)X_{\Gamma,\triv}$ for every $\gamma\in \Irr(\Aut\Gamma)$.
\item $\chi_{\Gamma,\gamma} = (\dim\gamma)\chi_{\Gamma,\triv}$ for every $\gamma\in \Irr(\Aut\Gamma)$.
\end{enumerate}
\end{proposition}

\begin{proof}
Recall that for each $\alpha\in \Comp_{\#V(\Gamma)}$ the coefficient of $M_\alpha$ in $X_{\Gamma}$ is $\#\Colour_{\Gamma,\alpha}$, while in $X_{\Gamma,\triv}$ this coefficient is $\#(\orbits{\Aut\Gamma}{\Colour_{\Gamma,\alpha}})$. Since an action of a finite group $G$ on a set $S$ is free if and only if $\# S = \# G \cdot \#(\orbits{G}{S})$, we conclude that (a) and (b) are equivalent. Applying the same considerations to the action of $\Aut\Gamma$ on the set $\Colour_{\Gamma,m}$ for each positive integer $m$ shows that (a) and (c) are equivalent.

For (d) and (e), note that if a finite group $G$ acts on a finite set $S$, then the permutation representation $\k^S$ of $G$ is isomorphic to a subrepresentation of $\k^G\otimes \k^{\#(\orbits{G}{S})}$, and is isomorphic to that whole representation if and only $G$ acts freely on $S$. We thus have, for every $\gamma\in \Irr(G)$,
\[
\dim\Hom_G(\k^S,\gamma) \leq \dim\Hom_G(\k^G\otimes \k^{\#(\orbits{G}{S})}, \gamma) = (\dim\gamma)\#(\orbits{G}{S}),
\]
with equality holding for every $\gamma$ if and only if $G$ acts freely on $S$. 

Applying this observation to the action of $\Aut\Gamma$ on  $\Colour_{\Gamma,\alpha}$ for each $\alpha\in \Comp_{\#V(\Gamma)}$ shows that the statements (a) and (d) are equivalent, and the same argument applied to the action of $\Aut\Gamma$ on each $\Colour_{\Gamma,m}$ shows that (a) and (e) are equivalent. 
\end{proof}

\begin{example}
The situation is  more interesting for the graphs with many symmetric colourings. To take an extreme example, consider for each $n\geq 0$ the graph $\overline{K}_n$ with $n$ vertices and no edges. We have
\[
X_{\overline{K}_n} = M_{(1)}^n = \left( \textstyle\sum_i x_i\right)^n.
\]
The additive subgroup $\Symreps$ of $\GraphAlg$ spanned by $\bigsqcup_{n\geq 0} \Irr(\Aut \overline{K}_n )$ is a Hopf subalgebra: it is the Hopf algebra of representations of the symmetric groups studied by Zelevinsky \cite[\S6]{Zelevinsky}. For each $n$ and each integer partition $\lambda$ of $n$, let $\gamma_\lambda$ be the irreducible representation of $S_n$ associated to $\lambda$, as explained for instance in \cite[Chapter 2]{James-Kerber}. Then we have 
\[
X_{\overline{K}_n,\gamma_\lambda} = s_{\lambda}
\]
with $s_\lambda$ being the Schur function (see, e.g., \cite[Chapter 7]{Stanley-EC2}) associated to $\lambda$. Indeed, the maps $\Psi :\gamma_\lambda\mapsto X_{\overline{K}_n,\gamma_\lambda}$ and $\Phi: \gamma_\lambda\mapsto s_\lambda$ are both Hopf algebra homomorphisms from $\Symreps$ to $\Sym_{\Z}$, and both maps yield the same character $\Symreps\to\Z$ upon specialisation at $(1,0,0,\ldots)$, and so the uniqueness in \cite[Theorem 4.3]{ABS} ensures that $\Psi=\Phi$. The decomposition of $X_{\overline{K}_n}$ given by part (4) of Theorem \ref{thm:GraphAlg} is thus the expansion of $\left(\sum x_i\right)^n$ in Schur functions.
\end{example}

\begin{example}
For our final example we  compute the symmetric functions $X_{\Gamma,\gamma}$ for the butterfly graph $\Gamma$:
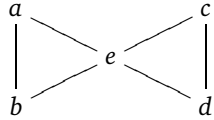
\[
\xygraph{ 
!{(0,0) }*+{e}="a" 
!{(1,.5) }*+{c}="b" 
!{(1,-.5) }*+{d}="c" 
!{(-1,.5)}*+{a}="d"
!{(-1,-.5)}*+{b}="e" 
"a"-"b"-"c"-"a"-"d"-"e"-"a"
}
\]
Stanley computed the chromatic symmetric function $X_\Gamma$ in \cite{Stanley-chromatic}.  Let $m_\alpha\in \Sym_{\Z}$ be the monomial symmetric function associated to the integer composition $\alpha$ (that is, the sum of the quasisymmetric functions $M_{\beta}$ over all integer compositions $\beta$ obtained from $\alpha$ by permuting the parts). In this notation, we have
\[
X_\Gamma = 4 m_{(2,2,1)} + 24 m_{(2,1,1,1)} + 120 m_{(1,1,1,1,1)}.
\]

Considering the complement $\overline{\Gamma}$ shows that the automorphism group of $\Gamma$ is isomorphic to the dihedral group $D_4$, generated by the $4$-cycle $r=(a\, d\, b\, c)$ and the involution $f=(a\, c)(b\, d)$. This group acts freely on $\Colour_{\Gamma,(2,1,1,1)}$ and $\Colour_{\Gamma,(1,1,1,1,1)}$, and so the arguments from Proposition \ref{prop:free-action} show that for each $\gamma\in \Irr(\Aut\Gamma)$ we have
\begin{equation*}\label{eq:bowtie-1}
X_{\Gamma,\gamma} = c_{\Gamma,\gamma} m_{(2,1,1)} + 3(\dim\gamma) m_{(2,1,1,1)} + 15(\dim\gamma) m_{(1,1,1,1,1)}
\end{equation*}
for some non-negative integer $c_{\Gamma,\gamma}$. 

To compute the coefficients $c_{\Gamma,\gamma}$ we must examine the permutation representation coming from the action of $\Aut\Gamma$ on $\Colour_{\Gamma,(2,2,1)}$. This action is transitive, but not free: for instance, the colouring
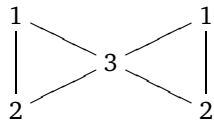
\[
\xygraph{ 
!{(0,0) }*+{3}="a" 
!{(1,.5) }*+{1}="b" 
!{(1,-.5) }*+{2}="c" 
!{(-1,.5)}*+{1}="d"
!{(-1,-.5)}*+{2}="e" 
"a"-"b"-"c"-"a"-"d"-"e"-"a"
}
\]
has isotropy $\{1,f\}$. We thus have $c_{\Gamma,\gamma} = \dim \gamma^f$,
the dimension of the space of $f$-fixed vectors in $\gamma$. These dimensions are easily computed from the character table of $\Aut\Gamma\cong D_4$, which we shall now recall. 

The group  $\Aut\Gamma$ has five conjugacy classes, represented by $1$, $r$, $r^2$, $f$, and $rf$. It has four $1$-dimensional representations---which we shall denote by $\triv$, $\chi_1$, $\chi_2$, and $\chi_3$---and a single irreducible $2$-dimensional representation $\rho$. The character table is as follows:
\begin{center}
\begin{tabular}{r | r  r  r r r }
$D_4$ & $1$ & $r$ & $r^2$ & $f$ & $rf$ \\
\hline 
$\triv$ & $1$ & $1$ & $1$ & $1$ & $1$ \\
$\chi_1$ & $1$ & $-1$ &  $1$ & $-1$ & $1$ \\
$\chi_2$ & $1$ & $-1$ & $1$ & $1$ & $-1$ \\
$\chi_3$ & $1$ & $1$ & $1$ & $-1$ & $-1$ \\
$\rho$ & $2$ & $0$ & $-2$ & $0$ & $0$ 
\end{tabular}
\end{center}
Computing the coefficients $c_{\Gamma,\gamma}=\dim \gamma^f$ from this table, we find:
\[
\begin{aligned}
X_{\Gamma,\triv}= X_{\Gamma,\chi_2} & = m_{(2,2,1)}  + 3m_{(2,1,1,1)} + 15 m_{(1,1,1,1,1)} \\
X_{\Gamma,\chi_1} = X_{\Gamma,\chi_3} & = 3m_{(2,1,1,1)} + 15 m_{(1,1,1,1,1)}  \\
X_{\Gamma,\rho} & = m_{(2,2,1)} +  6m_{(2,1,1,1)} + 30 m_{(1,1,1,1,1)}.
\end{aligned}
\]

\end{example}

\bibliographystyle{alpha}
\bibliography{Sn-mod-p}

\end{document}